\documentclass[conference,letterpaper]{IEEEtran}

\usepackage[utf8]{inputenc} 
\usepackage[T1]{fontenc}
\usepackage{url}
\usepackage{ifthen}
\usepackage{cite}
\usepackage[cmex10]{amsmath} 

\usepackage{amsfonts}
\usepackage{bbm}
\usepackage{xcolor}
\usepackage{amssymb}
\usepackage{amsthm}
\usepackage{mathtools}
\theoremstyle{definition}
\usepackage{autonum}

\DeclareMathOperator*{\argmax}{arg\,max}

\DeclareMathOperator*{\argsup}{arg\,sup}

\newtheorem{definition}{Definition}
\newtheorem{assumption}{Assumption}
\newtheorem{remark}{Remark}
\newtheorem{theorem}{Theorem}
\newtheorem{lemma}{Lemma}
\newtheorem{corollary}{Corollary}

\begin{document}
\title{Active Nonparametric Two-Sample Testing by Betting on Heterogeneous Data Sources} 

 \author{%
  \IEEEauthorblockN{Chia-Yu Hsu and Shubhanshu Shekhar}
 \IEEEauthorblockA{Department of Electrical Engineering and Computer Science\\
                    University of Michigan, 
                    Ann Arbor\\
                    Email: \{chiayuh, shubhan\}@umich.edu} }


\maketitle

\begin{abstract}
    We study the problem of active nonparametric sequential two-sample testing over multiple heterogeneous data sources. In each time slot, a decision-maker adaptively selects one of $K$ data sources and receives a paired sample generated from that source for testing. The goal is to decide as quickly as possible whether the pairs are generated from the same distribution or not. 
    The gain achieved by such adaptive sampling (in terms of smaller expected stopping time or larger error exponents) has been well-characterized for parametric models via Chernoff’s adaptive MLE selection rule \cite{Chernoff_59}. 
    However, analogous results are not known for the case of nonparametric problems, such as two-sample testing, where we place no restrictions on the distributions.   

    Our main contribution is a general active nonparametric testing procedure that combines an adaptive source-selecting strategy within the \emph{testing-by-betting} framework of~\cite{shekhar2025nonparametrictwosampletestingbetting} that works under minimal distributional assumptions.  
    In each time slot, our scheme proceeds by selecting a source according to a probability that mixes exploitation, favoring sources with the largest empirical distinguishability, and exploration via a vanishing greedy strategy. 
    The (paired) observations so collected are then used to update the ``betting-wealth process'', which is a stochastic process guaranteed to be a nonnegative martingale under the null. The procedure stops and rejects the null when the wealth process exceeds an appropriate threshold; an event that is unlikely under the null. We show that our test controls the type-I error at a prespecified level-$\alpha$ under the null, and establish its power-one property and a bound on its expected sample size under the alternative. Our results provide a precise characterization of the improvements achievable by a principled adaptive sampling strategy over its passive analog. 
    
     
\end{abstract}

\section{Introduction}
Heterogeneous data sources in sequential decision-making problems have received considerable attention, since access to multiple sources allows a decision-maker to reach a decision faster by adaptively sampling the most informative source. In the context of sequential hypothesis testing, Chernoff \cite{Chernoff_59} proposed the first adaptive selection strategy, which uses a maximum likelihood estimate (MLE) based on accumulated log-likelihood ratios to guide the choice of data source at each time slot, successfully demonstrating that the gain of adaptivity (so-called the gain of heterogeneous data sources in this paper) is a lower expected stopping time in an asymptotic regime with the same reliability. Following this work, many subsequent papers \cite{Hsu24,LanWang_21,NaghshvarJavidi_13,BaiGupta_17,BaiKatewa_15,Vaidhiyan2015,Prabhu2020} have extended Chernoff’s action-selection principle to various scenarios, including settings with expected budget constraints and temporally available data sources.

However, all these methods are ``likelihood-based'', and rely strongly on either the knowledge of the data distributions, or assume that the stochastic model is parametric. Such conditions are often violated in many practical real-world scenarios. The main motivation of our work is to initiate the study of the active sequential hypothesis testing problem in more general settings. 
As a first step in this direction, we base on \emph{sequential nonparametric two-sample testing} (SN2T), where the goal is to test whether two unknown data distributions are equal based only on sequentially collected paired samples without any parametric model assumptions. This setting naturally arises in applications where the underlying distributions are often complex and not well captured by simple parametric models. Prior works on SN2T \cite{shekhar2025nonparametrictwosampletestingbetting,Lheritier2018} have only developed sequential tests in the case of a \emph{single} data source (in this paper, one data source has a pair of distributions). To investigate the gain of adaptivity, we extend this single-source formulation to multiple data sources and a corresponding global null (and alternative).

Our active generalization of the sequential two-sample testing problem models many real-world problems. For example, consider an online A/B test on a large platform. Users are divided into several segments, such as different regions or device types. For each segment, the platform can observe how users interact under the current recommendation policy (version A) and under a new policy (version B), but the platform can only select one segment to get the observation per time slot due to some constraints, e.g., cost or bandwidth. The major question of this extension is \emph{how to exploit source heterogeneity to have the gain of adaptivity}.

Before proceeding to our results, we first briefly introduce our problem framework for the \emph{active} SN2T, so-called ASN2T in this paper. We assume that there are $K$ data sources and any data source, say source $k \in [K]$, consists of a pair of distributions $P_{k,1}$ and $P_{k,2}$, which are unknown. The global null hypothesis is that for any data source, its two distributions are the same. Then, in each time slot, the decision-maker can (adaptively) select one data source and collect paired samples drawn from the corresponding paired distributions. If the global null is false, it should stop and reject the null as soon as possible; if the global null is true, it wants to control the false alarm probability. Recall the online A/B test example, $P_{k,1}$ and $P_{k,2}$ may represent user-interaction distributions under two recommendation policies for segment $k$. This problem structure generalizes the single-source setting in the current literature on SN2T, where they consider $K=1$.

Our main result is that, following the “sequential tests of power one’’ criterion introduced by \cite{Darling1968}, we design an adaptive source-selection policy combined with the general framework of two-sample testing by betting \cite{shekhar2025nonparametrictwosampletestingbetting}. The resulting procedure (including the stopping rule and the adaptive selection rule) has the following three properties. First, it ensures the level-$\alpha$ property, i.e., under the global null, the probability of ever stopping is at most $\alpha$. Second, it has the power-one property, i.e., if the global null is false (when at least one data source has different distributions for its two samples, so-called global alternative), the probability of stopping in finite time is one. Third, under some additional conditions, we derive a nonasymptotic upper bound on the expected stopping time under the global alternative, showing that its order is at least as good as that in the single-source setting studied by \cite{shekhar2025nonparametrictwosampletestingbetting}. Moreover, when the global alternative holds and there exists a data source that has a better “distinguishability’’ than those of all other sources, our adaptive selection strategy concentrates sampling on this most informative source through a vanishing greedy exploration, leading to an asymptotically smaller expected stopping time than the single-source setting.

For the stopping rule, we follow the method of \cite{shekhar2025nonparametrictwosampletestingbetting}, which adapts an adaptive betting strategy (to make the betting-wealth process grow as fast as possible under the alternative) with an adaptive prediction strategy (to measure the empirical “distance’’ between distributions), providing a computationally and statistically powerful two-sample test. We show that, even with adaptive source selection, the associated betting-wealth process remains a nonnegative supermartingale, so Ville’s inequality \cite{ville1928} yields the level-$\alpha$ property. In addition, by applying the martingale strong law of large numbers, we can guarantee the power-one property under mild assumptions. The more interesting part is in the design of the selection rule. While our setting is similar to sequential hypothesis testing in multi-armed bandits \cite{Prabhu2020} and heterogeneous population models \cite{Adam2023}, these results cannot be applied to our problem. In our case, the “reward’’ induced by the adaptive prediction strategy is time-varying and non-i.i.d.. Moreover, we do not have parametric assumptions on the data distributions, and the decision-maker can sample only one data source at each time slot rather than all sources simultaneously. Building on the ideas of \cite{Hsu24}, we propose a simple forced exploration strategy with a vanishing exploration rate~(denoted by $\epsilon(t)$) as the time $t$ goes to infinity. The guidance of the selection is the \emph{empirical distance} induced by the prediction strategy sequentially updated for each data source. Specifically, at each time $t\ge1$, with probability $1-\epsilon(t)$, we select the source with the largest empirical distance at time $t-1$, and with the probability $\epsilon(t)$, we sample a source uniformly at random. By carefully choosing the decay of $\epsilon(t)$ (related to the gap between the best distinguishability and the second best distinguishability, so-called sub-optimality gap), we balance exploration and exploitation and satisfy the third property mentioned earlier. 

To the best of our knowledge, this is the first work to study active source selection in sequential nonparametric hypothesis testing, with a particular focus on two-sample testing. 

\section{Problem Formulation}\label{sec:problem_form}
Throughout, we assume that there are $K<\infty$ data sources, and we use the notation $[K]$ to represent the set $\left\{1,2,\ldots,K\right\}$.
\subsection{Statistical model and hypotheses}
For each data source $k\in[K]$, there are two distributions $P_{k,1}$ and $P_{k,2}$ (probably in a nonparametric distribution class). If a data source $k\in[K]$ is selected, the decision-maker receives $(X_{k,1}, X_{k,2})$ drawn from  the product distribution $P_{k,1} \times P_{k,2}$. Moreover, both $X_{k,1}$ and $X_{k,2}$ lie on a common alphabet $\mathcal{X}$~(which for simplicity we assume is the same across all $k \in [K]$).

For the hypotheses, there are global null $H_0$, which means for any data source $k\in[K]$, their corresponding paired distribution $P_{k,1}$ and $P_{k,2}$ are the same, and the global alternative $H_1$, which means there exists at least one $k\in[K]$, such that their corresponding paired distributions are different. That is 
\begin{align}
    &H_0:\forall\,k\in[K]\quad P_{k,1}=P_{k,2},\\
    &H_1:\exists k\in[K]\quad P_{k,1}\neq P_{k,2}.
\end{align}
For all $k\in[K]$, $P_{k,1}$ and $P_{k,2}$ are unknown to the decision-maker, and we make no further assumptions. We respectively denote $\mathbb{P}_{H_i}(\cdot)$ and $\mathbb{E}_{H_i}[\cdot]$ as the “probability’’ and the “expectation’’ under $H_i$ for $i\in\left\{0,1\right\}$.

\subsection{Metric for distance measure}
In this paper, we utilize a distance metric on the space of probability distributions utilized by \cite{shekhar2025nonparametrictwosampletestingbetting}. In particular, let $\mathcal{G}$ denote a class of test functions $g:\mathcal{X}\rightarrow[-1/2, 1/2]$. For any $k \in [K]$, we can use $\mathcal{G}$ to define a distance denoted by $D^{(k)}_{\mathcal{G}}$ between two distributions under $H_1$:
\begin{align}
    D^{(k)}_{\mathcal{G}}:=\sup_{g_k\in\mathcal{G}}\mathbb{E}_{H_1}[g_k(X_{k,1})-g_k(X_{k,2})]\quad\forall\,k\in[K],
\end{align}
and the corresponding maximizer is 
\begin{align}
    g^*_k:=\argsup_{g_k\in\mathcal{G}}\mathbb{E}_{H_1}[g_k(X_{k,1})-g_k(X_{k,2})]\quad\forall\,k\in[K],\label{eq:optimal_pred}
\end{align}
which is the so-called \emph{witness function}.
Different choices of the function class $\mathcal{G}$ lead to different distance measures~(or metrics), such as the $L_2$-norm or the kernel MMD metric.
Moreover, let $D^{(a^*)}_{\mathcal{G}}:=\max_{k\in[K]}D^{(k)}_{\mathcal{G}}$ and $a^*$ is the data source that attains $D^{(a^*)}_{\mathcal{G}}$. Then, the sub-optimality gap is defined as 
\begin{align}
    r:=D^{(a^*)}_{\mathcal{G}}-\max_{k\neq a^*}D^{(k)}_{\mathcal{G}}.
\end{align}
We assume the decision-maker knows a lower bound of $r$, defined as $L>0$ for the alternative hypothesis (see Remark~\ref{remark:known_lower_bound_of_gap}). To make our problem meaningful, we will work under the following assumptions on the function class $\mathcal{G}$.
\begin{assumption}[Positive sub-optimality gap]\label{a:positive_subopt_gap}
    $r>0$.
\end{assumption} 
\begin{assumption}[Global Distinguishability]\label{a:distinguishability}
    There exists $k\in[K]$ such that $D^{(k)}_{\mathcal{G}}>0$.
\end{assumption}
\begin{assumption}\label{a:symmetric_G}
    If $g\in\mathcal{G}$, then $-g\in\mathcal{G}$ too.
\end{assumption}
This last assumption is primarily to simplify the presentation, and we can relax this following the same ``hedging'' argument as discussed by~\cite{shekhar2025nonparametrictwosampletestingbetting}.

\subsection{Source-selection, stopping time, and performance measure}
At time slot $t\ge1$, the decision-maker selects a source $\delta_t\in[K]$ according to a (possibly randomized) selection policy, which can be fully adaptive. In other words, $\delta_t$ is a random variable whose distribution is determined by all information observed up to time $t-1$, including past actions, collected samples, or any functions or statistics that have been updated from these samples.

After taking the action, if $\delta_t = k$ for any source $k\in[K]$, the decision-maker receives a pair of observations $X_{k,1}(t)$ and $X_{k,2}(t)$ generated by $ P_{k,1}$ and $P_{k,2}$, respectively. That is, for any $t\ge1$,
\begin{align}
    X_{k,1}(t)\sim P_{k,1},\quad X_{k,2}(t)\sim P_{k,2}\quad\text{if }\delta_t=k\in[K].
\end{align}
We assume that, for the selected source $k$ at time $t$, the samples $X_{k,1}(t)$ and $X_{k,2}(t)$ are independent of each other. In addition, for each $t\ge1$ and given $\delta_t=k$, the pair $(X_{k,1}(t), X_{k,2}(t))$ is conditionally independent of all past information.

At the end of each time slot, the decision-maker uses all paired samples collected so far to decide whether to stop and reject $H_0$ or to continue sampling. The corresponding stopping time is denoted by $\tau$.

The procedure of the action-taking and stopping time is summarized below.
\begin{definition}[Active sequential two-sample test]\label{def:ASN2T}
    The procedure of the test is as follows: 
    \begin{itemize}
        \item At each time slot $t\ge1$, the distribution of the action $\delta_t$ is a function of the natural filtration $F_{t-1}$, where $F_{t-1} \coloneqq \sigma\left\{(X_{\delta_s,1}(s),X_{\delta_s,2}(s), \delta_s)_{s=1}^{t-1}\right\}$.
        \item The  event $\left\{\tau=t\right\}$ is $F_{t}$-measurable for all $t\ge1$.
    \end{itemize}
\end{definition}

Then, for the performance guarantee, we say that an active sequential two-sample test (as in Definition~\ref{def:ASN2T}) is a level-$\alpha$ power-one test if the probability of ever stopping under the null hypothesis is at most $\alpha$, and the probability of ever stopping under the alternative is equal to one. In short, a level-$\alpha$ power-one sequential test should satisfy
\begin{align}
    \mathbb{P}_{H_0}(\tau<\infty)\leq\alpha\quad\text{and}\quad\mathbb{P}_{H_1}(\tau<\infty)=1,
\end{align}
where $P_{H_i}$ means that the samples are drawn from distributions under hypothesis $H_i$. Furthermore, the expected stopping time is defined under $H_1$, i.e., we denote it as $\mathbb{E}_{H_1}[\tau]$.

Our goal is to construct a level-$\alpha$ power-one active sequential two-sample test in a fully nonparametric setting, and derive an upper bound on the expected stopping time, in the process precisely characterizing the gain of adaptive sampling.

In the following sections, to understand our proposed test, we begin by introducing an oracle test in our problem setting, following the idea in \cite{shekhar2025nonparametrictwosampletestingbetting}. Then, since the oracle test depends on the knowledge of data distributions, we propose our practical test, including the stopping and the selection rule. The stopping rule basically follows the testing-by-betting method proposed by \cite{shekhar2025nonparametrictwosampletestingbetting}, and the selection rule is based on a vanishing greedy mechanism. We then demonstrate our main theorem. Finally, we compare our expected stopping time with the passive setting in \cite{shekhar2025nonparametrictwosampletestingbetting} and identify conditions for the gain of adaptivity.

\begin{remark}[Knowledge of the lower bound of the sub-optimality gap]\label{remark:known_lower_bound_of_gap}
    Our policy uses a (vanishing) greedy exploration scheme to learn the distinguishability of each data source. When the sub-optimality gap is small, increased exploration is necessary to reliably identify the most informative source. While one may consider bandit exploration rules such as UCB \cite{lattimore2020}, standard stochastic-bandit methods do not directly apply here because our “reward’’ is time-varying, induced by the adaptive prediction strategy. 
    Finally, greedy is also easier to generalize to additional constraints, such as expected budget constraints or time-varying availability of data sources \cite{Hsu24}.
\end{remark}

\section{Oracle test}\label{sec:oracle_test}
In this section, we describe an ``oracle test'' that relies on the knowledge of the true (and unknown) distributions. While this test is clearly not practical, it serves as a template for the design of our practical test in the next section. 
In parametric settings, such as that considered by~\cite{Chernoff_59, NaghshvarJavidi_13, Prabhu2020}, the optimal approach is to design tests based on cumulative log-likelihood ratio processes. 
However, in our setting, the statistical model is nonparametric and unknown; hence, we follow the oracle test proposed by \cite{shekhar2025nonparametrictwosampletestingbetting} but with heterogeneous data sources.

The oracle test is associated with the optimal sampling policy with  $\delta_t=a^*$ for all $t\ge1$, and accumulates the difference $g^*_{a^*}(X_{a^*,1}(t))-g^*_{a^*}(X_{a^*,2}(t))$ over time slot $t\ge1$ via a log-optimal betting process. Specifically, we denote the oracle betting-wealth process as $\left\{W^*_t\right\}_{t\ge0}$, with $W^*_0=1$, and
\begin{align}
    &W^*_{t}=W^*_{t-1}\cdot(1+\lambda^*v^{(a^*)}_t(g^*_{a^*})), \; \forall \; t \geq 1,
\end{align}
where $\forall\,t\ge1$ and $g\in\mathcal{G}$, we use the notation
\begin{align}
    &v^{(k)}_t(g):=g(X_{k,1}(t))-g(X_{k,2}(t))\quad\forall\,k\in[K],\label{eq:v}\\
    &\lambda^*\in\argmax_{\lambda\in[-1,1]} \mathbb{E}\left[\log\left(1+\lambda(g^*_{a^*}(X_{a^*,1}) - g^*_{a^*}(X_{a^*,2})\right)\right],\nonumber\\
    &\text{with}\ (X_{a^*,1}, X_{a^*,2})\sim P_{X_{a^*,1}}\times P_{X_{a^*,2}}\nonumber.
\end{align}
We can then define the oracle stopping time as $\tau^*:=\inf\left\{t\ge1:W^*_t\geq1/\alpha\right\}$.
The term $\lambda^*$ is the optimal constant bet that ensures $\left\{W^*_t\right\}_{t\ge1}$ to (exponentially) grows with the fastest rate under $H_1$, leading to the smallest expected stopping time. Moreover, $g^*_{a^*}$ can be interpreted as the function that best distinguishes $P_{a^*,1}$ and $P_{a^*,2}$, and the per-time difference $v^{(a^*)}_t(g^*_{a^*})$ is from the most informative data source for the decision-maker under the alternative, implicitly showing the possibility of the gain of heterogeneous data sources. Also, it is easy to show that the oracle test $\tau^*$ satisfies the level-$\alpha$ power-one property by the argument in \cite{shekhar2025nonparametrictwosampletestingbetting}. We omit this discussion here.

However, since $\left\{g^*_k\right\}_{k\in[K]}$, $\lambda^*$, and $a^*$ require the knowledge of $\left\{(P_{k,1},P_{k,2})\right\}_{k\in[K]}$, this oracle betting-wealth process is not practical. Therefore, following \cite{shekhar2025nonparametrictwosampletestingbetting}, we utilize a data-driven method consisting of a \emph{prediction strategy} and a \emph{betting strategy} to approximate $\left\{g^*_k\right\}_{k\in[K]}$ and $\lambda^*$ by two $F_{t-1}$-measurable sequences, $\left\{g_{k,t}\right\}_{t\ge1}$ for each $k\in[K]$ and $\left\{\lambda_t\right\}_{t\ge1}$. For the selection strategy, we propose a data-driven randomized policy that balances \emph{exploration} and \emph{exploitation} well, which ensures the decision-maker can select the most informative data source $a^*$ most of the time. We discuss the details of our approach in the next section. 

\section{Main Results}\label{sec:main_results}
\subsection{Our proposed test}
For each $k\in[K]$, the update strategy of $\left\{g_{k,t}\right\}_{t\ge1}$ is denoted by $\mathcal{A}^{(k)}_{\text{pred}} := \left\{A^{(k)}_{\text{pred},t} : t \ge 1\right\}$, where each $A^{(k)}_{\text{pred},t}$ maps the past information $F_{t-1}$ to a function $g_{k,t} \in \mathcal{G}$. Similarly, the update strategy of $\left\{\lambda_t\right\}_{t\ge1}$ is denoted by $\mathcal{A}_\text{bet} := \left\{A_{\text{bet},t} : t \ge 1\right\}$, where each $A_{\text{bet},t}$ maps $F_{t-1}$ to a bet $\lambda_t \in [-1,1]$. Then, we follow similar updates as \cite{shekhar2025nonparametrictwosampletestingbetting} for these two strategies (see Remark~\ref{remark:update_of_the_betting}). More importantly, since the prediction strategy depends on the function class $\mathcal{G}$, the quality of the prediction function of each data source is represented by the \emph{individual} regret $R^{(k)}_t(\mathcal{A}^{(k)}_{\text{pred}})$ for each $k\in[K]$ and $t\ge1$, defined as 
\begin{align}
    R^{(k)}_t\equiv R^{(k)}_t(\mathcal{A}^{(k)}_{\text{pred}})\coloneqq &\sup_{g_k\in\mathcal{G}} \bigg( \sum_{s=1}^{t}v^{(k)}_s(g_{k})\mathbbm{1}\left\{\delta_s=k\right\} \nonumber\\
    &- \sum_{s=1}^{t}v^{(k)}_s(g_{k,s})\mathbbm{1}\left\{\delta_s=k\right\} \bigg),\label{eq:individual_regret}
\end{align}
where $\mathbbm{1}\left\{\cdot\right\}$ is the indicator function. In this paper, $R^{(k)}_t$ is equivalent to $R^{(k)}_t(\mathcal{A}^{(k)}_{\text{pred}})$ for notational simplicity.
This definition of regret extends the single-source setting in \cite{shekhar2025nonparametrictwosampletestingbetting}, where $K=1$. We will show the role of individual regret in our main theorem.

For our selection policy, after replacing $g^*_k$ by $\left\{g_{k,t}\right\}_{t\ge1}$ for each $k\in[K]$, we define the selection rule in terms of the resulting empirical distances over received samples. At time slot $t\ge1$, if source $\delta_t\in[K]$ is selected, the empirical distance in that slot is
\[
    v^{(\delta_t)}_{t}(g_{\delta_t,t})
    = g_{\delta_t,t}\bigl(X_{\delta_t,1}(t)\bigr)
      - g_{\delta_t,t}\bigl(X_{\delta_t,2}(t)\bigr).
\]
Based on these quantities, the distribution of $\delta_t$ is chosen according to a vanishing $\epsilon(t)$-greedy policy. In words, at each time slot $t\ge1$, with probability $1-\epsilon(t)$ the policy exploits by selecting the data source whose running average of the empirical distances is largest until time $t-1$, and with probability $\epsilon(t)$ it explores by selecting a source uniformly at random from $[K]$. The exploration rate satisfies $\epsilon(t)\in\Theta(1/t)$, which ensures that each data source is selected $\Omega(\log(t))$ times. This frequency is sufficient to learn the running average of the empirical distances of $a^*$ is larger than others. Hence, as $t$ is large enough, the probability of selecting $\delta_t=a^*$ approaches one. In addition, we choose the constant in $\epsilon(t)$ to be proportional to $ K/L^2$.
Consequently, when $L$ is small or $K$ is large, identifying $a^*$ is more difficult; hence, the policy naturally allocates more probability to exploration.

In short, our proposed test is summarized as follows: The data-driven betting-wealth process and our stopping time are defined as $\left\{W_t\right\}_{t\ge0}$ and $\tau^{(\text{ASN2T})}$. Set $W_0=1$ and $C\in O\left(\frac{K}{L^2}\right)$. Then, at each time slot $t\ge 1$,
\begin{align}
    &\epsilon(t)=\min\left\{1,C/t\right\},\\
    &\delta_t
    \begin{cases}
        =\argmax_{k\in[K]}\hat{\mu}_{k,{t-1}}\quad\text{with probability }1-\epsilon(t),\\
        \sim\text{Unif}([K])\quad\text{with probability }\epsilon(t),
    \end{cases}\\
    &W_{t}=W_{t-1}\cdot(1+\lambda_t v^{(\delta_t)}_t(g_{\delta_t,t})),
\end{align}
and $\tau^{(\text{ASN2T})}:=\inf\left\{t\ge1:W_t\geq1/\alpha\right\}$,
where the running average $\hat{\mu}_{k,t}$ is defined as, $\forall\,k\in[K]$ and $t\ge1$,
\begin{align}
    &\hat{\mu}_{k,t}:=\frac{1}{N_{k,t}}\sum_{s=1}^{t}v^{(k)}_s(g_{k,s})\mathbbm{1}\left\{\delta_s=k\right\},\\
    &N_{k,t}:=\sum_{s=1}^{t}\mathbbm{1}\left\{\delta_s=k\right\},\quad\hat{\mu}_{k,t}:=-\infty\,\text{if }N_{k,t}=0.
\end{align}
The exact constant $C$ is defined in  Appendix~\ref{a:proposed_test}.

\begin{remark}[Update of the prediction strategy and the betting function]\label{remark:update_of_the_betting}
    For the prediction strategy, we apply Online Gradient Ascent (OGA) for the prediction strategy as illustrated in \cite[Sections 3 and 4]{shekhar2025nonparametrictwosampletestingbetting}. Specifically, if a data source $k\in[K]$ is selected at a time slot $t$, its $g_{k,t}$ is updated by the collected samples drawn from this data source $k$ until time $t-1$ through OGA. For the betting function, Online Newton Strategy (ONS) \cite[Definition 5]{shekhar2025nonparametrictwosampletestingbetting} is applied, but replacing $v_t$ defined in by $v^{(\delta_t)}_{t}(g_{\delta_t,t})$ for all time slots $t\ge1$.
\end{remark}

\subsection{Main theorems}
Before stating our main theorem, we first present several lemmas that describe how the proposed test controls the action-taking process. These results will be used in the proof of our main theorem.

The following lemmas are built under Assumptions~\ref{a:positive_subopt_gap} to \ref{a:symmetric_G}. All detailed proof can be found in Appendix~\ref{a:useful_lemma}. 
\begin{lemma}[Each data source is sampled $\Omega(\log(t))$ times]\label{le:at_least_logt}
Under our selection policy, for any finite constant $c_1>0$, $\forall\,k\in[K]$,
    \begin{align}
        \mathbb{P}_{H_i}(N_{k,t}<c_1\log(t))\leq \frac{2}{t^5}\quad\forall\,t>T_1,i\in\{0,1\},
    \end{align} 
    where $T_1$ is a finite constant proportional to $K$ and $c_1$. 
\end{lemma}
Lemma~\ref{le:at_least_logt} is the intuitive result since our selection policy does uniform selection with probability $\Theta(1/t)$ at each time slot. 

Our next result builds upon~\ref{le:at_least_logt}. The high-level idea is that if a data source $k\in[K]$ is selected $\Omega(\log(t))$ times, its running average $\hat{\mu}_{k,t}$ will be close to the corresponding expected value $D^{(k)}_{\mathcal{G}}$ if the quality of the prediction is \emph{good}.

\begin{lemma}[Empirical mean estimation error]\label{le:empirical_mean_error}
If the regret of $a^*$ satisfies $R^{(a^*)}_t\in O(\sqrt{N_{a^*,t}})$, then
our proposed policy achieves
\begin{align}
    \mathbb{P}_{H_1}(\max_{k\neq a^*}\hat{\mu}_{k,t}\geq\hat{\mu}_{a^*,t})\in O\left( \frac{1}{t^5}\cdot\frac{K}{r^2} \right)\quad\forall\,t>T_2,
\end{align}
where $T_2>0$ is a finite constant proportional to $K$ and $1/L$.
\end{lemma}
Lemma~\ref{le:empirical_mean_error} demonstrates that as time is large enough, our selection policy selects the optimal data sources with a high probability. The condition on the regret of the prediction strategy can be achieved by a direct application of the regret bound of OGA \cite[Appendix A.4]{shekhar2025nonparametrictwosampletestingbetting}.

\begin{lemma}[Sub-optimal source is drawn at most $O(\sqrt{t})$ times]\label{le:sub-opt_source_is_drawn_at_most_n0.5}
By Lemma~\ref{le:empirical_mean_error}, for all $k\neq a^*$, if the regret of the prediction strategy on $a^*$ satisfies $R^{(a^*)}_t\in O(\sqrt{N_{a^*,t}})$, then
\begin{align}
    \mathbb{P}_{H_1}\left(N_{k,t}>\sqrt{t}\right)\in O\left(\frac{1}{t^2}\cdot\frac{K}{r^2}\right)\quad\forall\,t>T_3,
\end{align}
where $T_3>0$ is a finite constant proportional to $K$ and $1/L$.
\end{lemma}
Lemma~\ref{le:sub-opt_source_is_drawn_at_most_n0.5} is an application of Lemma~\ref{le:empirical_mean_error}.
With these three Lemmas, we can view that our proposed test \emph{usually} receives samples drawn from the data source $a^*$. Then, our main theorem is as follows: We use $\forall\,H_i$ to denote the class of all distributions that satisfy the hypothesis $H_i$ for $i\in\left\{0,1\right\}$.
\begin{theorem}\label{thm:main_thm}
    The proposed test satisfies the following:
    \begin{itemize}
        \item Level-$\alpha$ property: Under Assumption~\ref{a:symmetric_G}, for any $\left\{\mathcal{A}^{(k)}_{\text{pred}}\right\}_{k\in[K]}$, $\mathcal{A}_{\text{bet}}$, and selection policy defined in Definition~\ref{def:ASN2T}, given a predefined accuracy $\alpha\in(0,1)$, 
            \begin{align}
            \mathbb{P}_{H_0}(\tau^{(\text{ASN2T})}<\infty)\leq\alpha\quad\forall\,H_0.
            \end{align}
        \item Power-one property: Under Assumptions~\ref{a:positive_subopt_gap} to \ref{a:symmetric_G}, if
            \begin{align}
                \limsup_{t\rightarrow\infty}\left(\frac{R^{(a^*)}_{t}}{t}-\frac{1}{t}\sum_{s=1}^{t}D^{(a^*)}_{\mathcal{G}}\mathbbm{1}\left\{\delta_s=a^*\right\}\right) \stackrel{\text{a.s.}}{<}0,\label{eq:condition_for_power-one}
            \end{align}
            then
            \begin{equation}
                \mathbb{P}_{H_1}(\tau^{(\text{ASN2T})}<\infty)=1\quad\forall H_1.
            \end{equation}
        \item Expected stopping time: Under Assumptions~\ref{a:positive_subopt_gap} to \ref{a:symmetric_G}, if there exists  $\left\{r_{t}\right\}_{t\ge1}$ such that $r_t\overset{t\rightarrow\infty}{\rightarrow}0$ and $\sum_{t\ge1}\mathbb{P}(E^c_{t,a^*})<\infty$ where $E_{t,a^*}:=\left\{R^{(a^*)}_{t}/t\leq r_{t}\right\}$, then
                \begin{align}
                    \mathbb{E}_{H_1}[\tau^{(\text{ASN2T})}]\in O\left(t_0+\sum_{t\geq 1}\mathbb{P}(E^c_{t,a^*})+T_{\text{explore}}\right)
                \end{align}
                where
                \begin{align}
                    &t_0:=\inf\left\{ t\ge1:D^{(a^*)}_{\mathcal{G}}\geq r_{t}+\sigma_{a^*}\sqrt{\frac{\log\frac{t}{\alpha}}{t}}+\frac{\log \frac{t}{\alpha}}{t}\right\},\\
                    &\sigma_{k}:=\sqrt{\sup_{g_{k}}\text{Var}(g_{k}(X_{k,1})-g_{k}(X_{k,2}))}\quad\forall\,k\in[K],
                \end{align}
                and $T_{\text{explore}}$ denotes the cost of exploration, which is proportional to $K$ and $1/L$.
    \end{itemize}
\end{theorem}
The proof of Theorem~\ref{thm:main_thm} is in Appendix~\ref{a:proof_of_main_thm}. We incorporate the active selection with the ideas in \cite{shekhar2025nonparametrictwosampletestingbetting}. Specifically,
the level-$\alpha$ property is a direct application of Ville's inequality since $\left\{W_t\right\}_{t\ge0}$ is a non-negative martingale. The power-one property follows as a consequence of a martingale strong law of large numbers argument, where the condition \eqref{eq:condition_for_power-one} can be achieved since $R^{(a^*)}_{t}(\mathcal{A}^{a^*}_{\text{pred}})/t$ converges to zero as $t\rightarrow\infty$ almost surely (by the regret of OGA \cite[Appendix A.4]{shekhar2025nonparametrictwosampletestingbetting}), $N_{a^*,t}/t$ converges to one as $t\rightarrow\infty$ almost surely (from Lemma~\ref{le:sub-opt_source_is_drawn_at_most_n0.5}), and $D^{(a^*)}_{\mathcal{G}}>0$. 

More interestingly, the expected stopping time obtained here may be smaller than the passive setting \cite{shekhar2025nonparametrictwosampletestingbetting} in some cases. In applications considered in \cite[Section 3 and 4]{shekhar2025nonparametrictwosampletestingbetting}, one can verify that the quantity $t_0$ defined in this paper and the term $\sum_{t\geq 1}\mathbb{P}(E^c_{t,a^*})+T_{\text{explore}}\in O(1)$ are both of order $O(\sigma^2_{a^*}\log(1/D^{(a^*)}_{\mathcal{G}}\alpha)/(D^{(a^*)}_{\mathcal{G}})^2)$. Consequently, when $\alpha$ is small enough so that this term dominates the expected stopping time, and when $\sigma_{a^*}^2 \leq \max_{k \neq a^*} \sigma_k^2$, our proposed test achieves a smaller expected stopping time than the procedure in \cite{shekhar2025nonparametrictwosampletestingbetting} whenever their (single) data source does not coincide with the most informative source 
$a^*$ in our setting. This demonstrates the gain from heterogeneity (so-called the gain of adaptivity). We summarize this comparison in Corollary~\ref{co:comparison}. 

\begin{corollary}[Gain of adaptivity]\label{co:comparison}
If the passive setting \cite{shekhar2025nonparametrictwosampletestingbetting} uses a sub-optimal data source and $\sigma_{a^*}^2\leq\max_{k\neq a^*}\sigma_{k}^2$,  
then
\begin{align}
        \mathbb{E}_{H_1}[\tau^{\text{ASN2T}}]\leq\mathbb{E}_{H_1}[\tau^{(p)}]+O\left(\frac{K^2}{r^2}(\log(1/\alpha))^{2/3}\right),\label{eq:comparison_upperbound}
    \end{align}
    where $\tau^{(p)}$ denotes the stopping time in the general test proposed in \cite{shekhar2025nonparametrictwosampletestingbetting}.
    Asymptotically,  
    \begin{align}
        \lim_{\alpha\rightarrow0}\frac{\mathbb{E}[\tau^{\text{ASN2T}}]}{\log(1/\alpha)}\leq\lim_{\alpha\rightarrow0}\frac{\mathbb{E}[\tau^{(p)}]}{\log(1/\alpha)}.
    \end{align}
\end{corollary}
The additional cost in the upper bound \eqref{eq:comparison_upperbound} is due to exploration. The detailed proof of the corollary is in Appendix~\ref{a:proof_of_comparison}. 

\section{Conclusion}
In this paper, we introduce active source selection into general nonparametric sequential two-sample testing. Our proposed test combines the powerful testing-by-betting framework of \cite{shekhar2025nonparametrictwosampletestingbetting} with a simple yet effective $\epsilon(t)$-greedy exploration strategy. Our test is a level-$\alpha$ power-one procedure, and we show that adaptive source selection yields a provable gain in terms of the expected stopping time. To the best of our knowledge, this is the first work that studies active selection in sequential hypothesis testing in a nonparametric model (without the knowledge of data distribution). Our selection policy may also be adapted to more general sequential nonparametric testing problems. Finally, our setting naturally suggests a new \emph{best-arm identification problem} with time-varying rewards, where each arm has an unknown reward distribution that somehow converges over time.

\newpage
\IEEEtriggeratref{9}
\bibliographystyle{IEEEtran}
\bibliography{A2T}

\onecolumn 
\section{Appendix}
In this Appendix, we first provide details of our proposed test in Appendix~\ref{a:proposed_test}. The ONS betting strategy in Appendix~\ref{a:ons_betting}, which is similar to that in \cite{shekhar2025nonparametrictwosampletestingbetting}. Following this, the detailed proofs of Lemma~\ref{le:at_least_logt}, Lemma~\ref{le:empirical_mean_error}, and Lemma~\ref{le:sub-opt_source_is_drawn_at_most_n0.5} are shown in Appendix~\ref{a:useful_lemma}. The proof of Theorem~\ref{thm:main_thm} is provided in Appendix~\ref{a:proof_of_main_thm}, where we mainly follow the proof idea in \cite[Theorem 1]{shekhar2025nonparametrictwosampletestingbetting} with careful handling of the action-taking behavior. Then, we provide the detailed proof of Corollary~\ref{co:comparison} in Appendix~\ref{a:proof_of_comparison}. Finally, we list technical lemmas in Appendix~\ref{a:technical_lemmas} for the above proofs.

\subsection{Proposed test}\label{a:proposed_test}
Our proposed test is summarized as follows: The data-driven betting-wealth process and our stopping time are defined as $\left\{W_t\right\}_{t\ge0}$ and $\tau^{(\text{ASN2T})}$. Set $W_0=1$ and
\begin{align}
    C:=K(c_1 + 40K) + 4K \sqrt{5K(c_1 + 20K)},\quad c_1=\frac{2560}{L^2},
\end{align}
where $L$ is the lower bound of the suboptimal gap $r$ defined in Section~\ref{sec:problem_form}. Then,  at each time slot $t\ge1$,
\begin{align}
    &\epsilon(t)=\min\left\{1,C/t\right\},\\
    &\delta_t
    \begin{cases}
        =\argmax_{k\in[K]}\hat{\mu}_{k,{t-1}}\quad\text{with probability }1-\epsilon(t),\\
        \sim\text{Unif}([K])\quad\text{with probability }\epsilon(t),
    \end{cases}\\
    &W_{t}=W_{t-1}\cdot(1+\lambda_t v^{(\delta_t)}_t(g_{\delta_t,t})),
\end{align}
and $\tau^{(\text{ASN2T})}:=\inf\left\{t\ge1:W_t\geq1/\alpha\right\}$,
where $v^{(k)}_t(g)$ and the running average $\hat{\mu}_{k,t}$ are defined as, $\forall\,k\in[K], t\geq1, g\in\mathcal{G}$,
\begin{align}
    &v^{(k)}_t(g):=g(X_{k,1}(t))-g(X_{k,2}(t)),\\
    &\hat{\mu}_{k,t}:=\frac{1}{N_{k,t}}\sum_{s=1}^{t}v^{(k)}_s(g_{k,s})\mathbbm{1}\left\{\delta_s=k\right\},\\
    &N_{k,t}:=\sum_{s=1}^{t}\mathbbm{1}\left\{\delta_s=k\right\}.
\end{align}

\subsection{ONS betting strategy}\label{a:ons_betting}
In this section, we provide the ONS betting strategy for updating $\lambda_{t}$ defined in our proposed test. Recall that
\begin{align}
    v^{(k)}_t(g):=g(X_{k,1}(t))-g(X_{k,2}(t))\quad\forall\,k\in[K], t\ge1, g\in\mathcal{G},
\end{align}
we have the following definition.
\begin{definition}(ONS betting strategy)
    Initialize $\lambda_1 = 0$, and $a_0 = 1$. Then, for $t = 1, 2, \ldots$:
\begin{itemize}
    \item Observe $v^{(\delta_t)}_{t}(g_{\delta_t,t}) \in [-1,1]$.
    \item Set $z_t = v^{(\delta_t)}_{t}(g_{\delta_t,t}) / (1 + v^{(\delta_t)}_{t}(g_{\delta_t,t}) \lambda_t)$.
    \item Update $a_t = a_{t-1} + z_t^2$.
    \item Update $\lambda_{t+1}$ as follows:
    \[
        \lambda_{t+1} = \min\left\{ \frac{1}{2}, \max\left\{ -\frac{1}{2}, \lambda_t - \frac{2}{2 - \log 3} \cdot \frac{z_t}{a_t} \right\} \right\}.
    \]
\end{itemize}
\end{definition}
Therefore, following the argument in \cite{Cutkosky2018}, the lower bound of the growth rate of $W_t$ in our proposed test is as follows,
\begin{align}
    W_t&\geq \frac{1}{\sum_{s=1}^t (v^{(\delta_s)}_{s}(g_{\delta_s,s}))^2} \exp\left( \frac{\left( \sum_{s=1}^t v^{(\delta_s)}_{s}(g_{\delta_s,s}) \right)^2}{4 \left( \sum_{s=1}^t (v^{(\delta_s)}_{s}(g_{\delta_s,s}))^2 + \sum_{s=1}^t v^{(\delta_s)}_{s}(g_{\delta_s,s}) \right)} \right)\\
    &\geq \exp\left( \frac{t}{8} \left( \frac{1}{t} \sum_{s=1}^t v^{(\delta_s)}_{s}(g_{\delta_s,s}) \right)^2 - \log(t) \right).
\end{align}
One can notice that under the alternative, if $\delta_s=a^*$, the lower bound of the growth rate of $W_t$ is intuitively the fastest.

\subsection{Proof of lemmas in the paper}\label{a:useful_lemma}
The lemmas introduced in the paper are for the analysis of the expected stopping time, so it is sufficient only to consider the alternative is true. Hence, we may ignore the index $H_1$ in $\mathbb{P}_{H_1}$ and $\mathbb{E}_{H_1}$ for notational simplicity in these lemmas and their relatives.
\subsubsection{Proof of Lemma~\ref{le:at_least_logt}}\label{a:proof_of_at_least_logt}

    Recall $N_{k,t}:=\sum_{s=1}^{t}\mathbbm{1}\left\{\delta_s=k\right\}$. We have, $\forall\,k\in[K]$, and any finite $c_1>0$,
    \begin{align}
        \mathbb{P}\left(N_{k,t}<c_1\log(t)\right)&=\mathbb{P}\left(\sum_{s=1}^{t}\mathbbm{1}\left\{\delta_s=k\right\}<c_1\log(t)\right)\\
        &=\mathbb{P}\left(\sum_{s=1}^{t}\left(\mathbbm{1}\left\{\delta_s=k\right\}-P(\delta_s=k|F_{s-1})\right)<c_1\log(t)-\sum_{s=1}^{t}P(\delta_s=k|F_{s-1})\right),\label{eq:lemma1_1}
    \end{align}
    where \eqref{eq:lemma1_1} is by subtracting $\sum_{s=1}^{t}P(\delta_s=k|F_{s-1})$ on both sides. We will apply Lemma~\ref{le:bernstein} to propose an upper bound.
    
    For notational simplicity, we define $a_s:=\mathbbm{1}\left\{\delta_s=k\right\}-P(\delta_s=k|F_{s-1})$. It is easy to verify that $\left\{a_s, F_{s}\right\}_{s\ge0}$ is a martingale difference. Then, by our selection policy, 
     \begin{align}
        P(\delta_s=k|F_{s-1})=(1-\epsilon(s))\mathbbm{1}\left\{\hat{\mu}_{k,s-1}\geq\max_{m\neq k}\hat{\mu}_{m,{s-1}}\right\}+\epsilon(s)\frac{1}{K},
    \end{align}
    we know $P(\delta_s=k|F_{s-1})\geq\frac{1}{K}\epsilon(s)$, where $\epsilon(s):=\min\left\{1,C/s\right\}$, leading to
    \begin{align}
        \sum_{s=1}^{t}P(\delta_s=k|F_{s-1})\geq\frac{1}{K}\sum_{s=1}^{t}\epsilon(s)=\frac{1}{K}\left(C+\sum_{s=C+1}^{t}\frac{C}{s}\right)\geq\frac{C}{K}\left(\log(t)+1-\log(C)\right),
    \end{align}
    where the last inequality is due to $\sum_{s=C+1}^{t}1/s\geq \log(t)-\log(C)$ and $t\geq C+1$.
    Plugging the definition of $a_s$ and this upper bound into \eqref{eq:lemma1_1}, we have 
    \begin{align}
        \mathbb{P}\left(N_{k,t}<c_1\log(t)\right)&\leq\mathbb{P}\left(\sum_{s=1}^{t}a_s<c_1\log(t)-\frac{C}{K}\left(\log(t)+1-\log(C)\right)\right)\\
        &=\mathbb{P}\left(\sum_{s=1}^{t}a_s<\left(c_1-\frac{C}{K}\right)\log(t)-\frac{C}{K}\left(1-\log(C)\right)\right),\label{eq:lemma1_2}
    \end{align}
    where in \eqref{eq:lemma1_2}, we put $\log(t)$ together.
    Then, to apply Lemma~\ref{le:bernstein}, we know $|a_s|\leq 1$ for all $s\ge1$. It remains to ensure that the right-hand side of \eqref{eq:lemma1_2} is negative and to upper bound $\sum_{s=1}^{t}\mathbb{E}[a_s^2|F_{s-1}]$. 
    
    First of all, we set $C>Kc_1$ (we will provide the overall requirement of $C$ later) and have
    \begin{align}
        \sum_{s=1}^{t}\mathbb{E}[a_s^2|F_{s-1}]&=\sum_{s=1}^{t}\mathbb{E}[\mathbbm{1}\left\{\delta_s=k\right\}^2|F_{s-1}]-P(\delta_s=k|F_{s-1})^2\label{eq:lemma1_3}\\
        &=\sum_{s=1}^{t}\mathbb{E}[\mathbbm{1}\left\{\delta_s=k\right\}|F_{s-1}]-P(\delta_s=k|F_{s-1})^2\label{eq:lemma1_4}\\
        &=\sum_{s=1}^{t}P(\delta_s=k|F_{s-1})(1-P(\delta_s=k|F_{s-1}))\label{eq:lemma1_5}\\
        &\leq\sum_{s=1}^{t}\min\left\{P(\delta_s=k|F_{s-1}), 1-P(\delta_s=k|F_{s-1})\right\} \label{eq:lemma1_6}
    \end{align}
    where \eqref{eq:lemma1_3} is from the definition of $a_s$, \eqref{eq:lemma1_4} holds true since $\mathbbm{1}\{\delta_s=k\}$ is a binary value, \eqref{eq:lemma1_5} is from the definition of the expectation, and \eqref{eq:lemma1_6} is due to $ab\leq \min\left\{a,b\right\}$ if $a,b\in[0,1]$. Moreover, by our selection policy,
    we know $P(\delta_s=k|F_{s-1})$ is either $(1-(K-1)\epsilon(s)/K)$ or $\epsilon(s)/K$ at each time slot $s\ge1$, hence, $\min\left\{P(\delta_s=k|F_{s-1}), 1-P(\delta_s=k|F_{s-1})\right\}$ equals to
    \begin{align}
        \min\{(1-\frac{K-1}{K}\epsilon(s)), \frac{K-1}{K}\epsilon(s)\}\quad\text{or}\quad\min\{\frac{1}{K}\epsilon(s), 1-\frac{1}{K}\epsilon(s)\}.
    \end{align}
    Then, 
    by $\epsilon(s)=\min\{1,C/s\}\leq\frac{1}{2}\frac{K}{K-1}$, where we can assume $K\ge2$ for the heterogeneous data sources, we have
    \begin{align}
        \min\left\{P(\delta_s=k|F_{s-1}), 1-P(\delta_s=k|F_{s-1})\right\}\leq\frac{K-1}{K}\epsilon(s)\quad\forall\,s\ge1.
    \end{align}
    Therefore, we get 
    \begin{align}
        \sum_{s=1}^{t}\mathbb{E}[a_s^2|F_{s-1}]&\leq C+\sum_{s=C+1}^{t}\frac{C}{s}\frac{K-1}{K} \label{eq:lemma1_7}\\
        &\leq C+ \frac{C(K-1)}{K}(\log(t)-\log(C))\label{eq:lemma1_8}\\
        &\leq\frac{C(K-1)}{K}\log (t)+C\label{eq:lemma1_9},
    \end{align}
    where \eqref{eq:lemma1_7} is from the definition of $\epsilon(s)$, \eqref{eq:lemma1_8} is from $\sum_{s=C+1}^{t}1/s\leq\log(t)-\log(C)$ and $t\geq C+1$, and \eqref{eq:lemma1_9} holds if $C\ge1$ (we will set $C$ later).

    Then, we can apply Lemma~\ref{le:bernstein}. For notational simplicity, define $A:=\frac{C}{K} - c_1$ and $B:=\frac{C}{K}(\log(C)-1)$
    and $C>\max\{1,Kc_1\}$, then for all $t>\max\{C+1,\exp\left( \frac{ B}{ A}\right)\}$, i.e., $A\log(t)-B>0$,
    \begin{align}
        \mathbb{P}\left(N_{k,t}<c_1\log(t)\right)&\leq\mathbb{P}\left(\sum_{s=1}^{t}a_s<-(A\log(t)-B)\right)\\
        &\leq2\exp\left(-\frac{1}{2}\frac{(A\log(t)-B)^2}{\frac{C(K-1)}{K}\log (t)+C+\frac{A\log(t)-B}{3}}\right),
    \end{align}
    where the last inequality is from Lemma~\ref{le:bernstein}.
    
    In the following, we simplify the exponent. 
    For the numerator, $A\log(t)-B\geq\frac{A}{2}\log(t)$, $\forall\,t\ge \exp(2B/A)$. Moreover,  the upper bound of the denominator is derived as follows
    \begin{align}
        \frac{C(K-1)}{K}\log (t)+C+\frac{A\log(t)-B}{3}&\leq\frac{C(K-1)}{K}\log (t)+C+\frac{A}{3}\log(t)\label{eq:lemma1_11}\\
        &=\left(\frac{C(K-1)}{K}+\frac{A}{3}\right)\log(t)+C\label{eq:lemma1_12}\\
        &=\left(\frac{3CK-2C}{3K}-\frac{c_1}{3}\right)\log(t)+C\label{eq:lemma1_13}\\
        &\leq C\log(t)+C\label{eq:lemma1_14},
    \end{align}
    where \eqref{eq:lemma1_11} holds if $C>2$, i.e., $B>0$, in \eqref{eq:lemma1_12} we put $\log(t)$ together, \eqref{eq:lemma1_13} is from the definition of $A$, and \eqref{eq:lemma1_14} is from $C\geq0$ and $c_1\geq0$.
    Then, $\forall\,t\geq 2$, $C\leq C\log(t)$, leading to 
    \begin{align}
        \frac{C(K-1)}{K}\log (t)+C+\frac{A\log(t)-B}{3}\leq 2C\log(t).
    \end{align}
    
    Consequently, combining the above results, we have, $\forall\,t>\max\left\{C+1, 2, \exp(2B/A)\right\}$, $C>\max\{Kc_1, 2\}$,
    \begin{align}
        \mathbb{P}\left(N_{k,t}<c_1\log(t)\right)&\leq2\exp\left(-\frac{1}{2}\frac{\frac{A^2}{4} \log^2(t)}{2C\log(t)}\right)=2\exp\left(-\frac{A^2}{16C}\log(t)\right)=\frac{2}{t^{A^2/(16C)}}=\frac{2}{t^{(C/K - c_1)^2/16C}}.
    \end{align}
    Then, we further set $C \geq K(c_1 + 40K) + 4K \sqrt{5K(c_1 + 20K)}$, leading to 
    \begin{align}
        \mathbb{P}\left(N_{k,t}<c_1\log(t)\right)\leq\frac{2}{t^{5}}.
    \end{align}

    To conclude this proof, we get for any $c_1>0$, if $C \geq K(c_1 + 40K) + 4K \sqrt{5K(c_1 + 20K)}$, then
    \begin{align}
         \mathbb{P}\left(N_{k,t}<c_1\log(t)\right)\leq\frac{2}{t^{5}}\quad\forall\,t>T_1:=t>\max\left\{C+1, \exp(2B/A)\right\}.
    \end{align} 
    where $A:=\frac{C}{K} - c_1$ and $B:=\frac{C}{K}(\log(C)-1)$. One can notice that the $T_1\in O(C^{3})$. Moreover, this lemma only depends on the design of the exploration. Therefore, the inequality also holds for the null hypothesis. 

\subsubsection{Proof of Lemma~\ref{le:empirical_mean_error}}\label{sec:proof_of_empirical_mean_error}
The goal is to show that there exists a finite constant time $T_2>0$ such that 
\begin{align}
    \mathbb{P}(\max_{k\neq a^*}\hat{\mu}_{k,t}\geq\hat{\mu}_{a^*,t})\in O(\frac{K}{t^5}\frac{1}{r^2})\quad \forall\,t>T_2,
\end{align}
under the sublinear regret condition, $R^{(a^*)}_t\in O\left(\sqrt{N_{a^*,t}}\right)$.

Recall $r:=D^{(a^*)}_{\mathcal{G}}-\max_{k\neq a^*}D^{(k)}_{\mathcal{G}}$.
Since
\begin{align}
    \left\{\max_{k\neq a^*}\hat{\mu}_{k,t}\geq\hat{\mu}_{a^*,t}\right\}&=\left\{D^{(a^*)}_{\mathcal{G}}-\hat{\mu}_{a^*,t}+\max_{k\neq a^*}\hat{\mu}_{k,t}-\max_{k\neq a^*}D^{(k)}_{\mathcal{G}}\ge r\right\}\label{eq:lemma2_1}\\
    &\subseteq\left\{D^{(a^*)}_{\mathcal{G}}-\hat{\mu}_{a^*,t}\geq\frac{r}{2}\right\}\cup\left\{\max_{k\neq a^*}\hat{\mu}_{k,t}-\max_{k\neq a^*}D^{(k)}_{\mathcal{G}}\geq \frac{r}{2}\right\}\label{eq:lemma2_2}\\
    &\subseteq\left\{D^{(a^*)}_{\mathcal{G}}-\hat{\mu}_{a^*,t}\geq\frac{r}{2}\right\}\cup\left\{\cup_{k\neq a^*}\left\{\hat{\mu}_{k,t}-D^{(k)}_{\mathcal{G}}\geq\frac{r}{2}\right\}\right\},\label{eq:lemma2_3}
\end{align}
where \eqref{eq:lemma2_1} is from the definition of $r$, \eqref{eq:lemma2_2} is due to $\{a+b\geq r\}\subseteq\{a\geq r/2 \cup b\geq r/2\}$, and \eqref{eq:lemma2_3} is due to $\max_k(a_k-b_k)\geq\max_k a_k-\max_k b_k$ and the definition of $\text{max}$. Then, by union bound, we have 
\begin{align}
\mathbb{P}(\max_{k\neq a^*}\hat{\mu}_{k,t}\geq\hat{\mu}_{a^*,t})\leq\mathbb{P}\left(D^{(a^*)}_{\mathcal{G}}-\hat{\mu}_{a^*,t}\geq\frac{r}{2}\right)+\sum_{k\neq a^*}\mathbb{P}\left(\hat{\mu}_{k,t}-D^{(k)}_{\mathcal{G}}\geq\frac{r}{2}\right).
\end{align}
Consequently, by Lemma~\ref{le:deviation_of_esti_source1} and Lemma~\ref{le:deviation_of_esti_source2},
if $\exists M>0$ such that
\begin{align}
    \frac{R^{(a^*)}_t}{N_{a^*,t}}\leq \frac{M}{\sqrt{N_{a^*,t}}}\quad\text{almost surely},
\end{align}
we have
\begin{align}
    \mathbb{P}(\max_{k\neq a^*}\hat{\mu}_{k,t}\geq\hat{\mu}_{a^*,t})&\leq \frac{K}{1-\exp\left\{-\frac{r^2}{512}\right\}}\frac{1}{t^{5}}+\frac{2K}{t^{5}}\\
    &\leq \frac{K}{t^5}(1+\frac{512}{r^2})+\frac{2K}{t^5}\quad\forall t\geq \max\left\{T_1, T_{2,1}\right\}\label{eq:lemma2_4}
\end{align}
where \eqref{eq:lemma2_4} is due to $1/(1-\exp(-x))\leq1+1/x$ for $x>0$, and $T_1$ is defined in Appendix~\ref{a:proof_of_at_least_logt} with $c_1\geq\frac{2560}{r^2}$. Therefore, we have $T_2:=\max\left\{T_1, T_{2,1}\right\}$ with $T_{2,1}:=\exp(16M^2/(r^2 c_1))$. 

\subsubsection{Proof of Lemma~\ref{le:sub-opt_source_is_drawn_at_most_n0.5}}\label{a:sub-opt source is drawn at most sqrt{n}}
The goal is to show that there exists a finite $T_3>0$ such that $\forall\,t\geq T_3$,
\begin{align}
    \mathbb{P}\left(N_{k,t}>\sqrt{t}\right)\in O(\frac{K}{t^2}\cdot\frac{1}{r^2}),
\end{align}

    For any $k\neq a^*$, let us define 
    $a_s=\mathbbm{1}\left\{\delta_s=k\right\}-P(\delta_s=k|F_{s-1})$ and from Appendix~\ref{a:proof_of_at_least_logt}, it is a martingale process adapted to $\left\{F_{s}\right\}_{s\ge1}$. By the definition of $N_{k,t}$, we have 
    \begin{align}
        \mathbb{P}\left(N_{k,t}>\sqrt{t}\right)&\leq\mathbb{P}\left(\sum_{s=1}^{t}a_s>\sqrt{t}-\sum_{s=1}^{t}P(\delta_s=k|F_{s-1})\right).
    \end{align}
    Then, by our selection policy, $P(\delta_s=k|F_{s-1})=(1-\epsilon(s))\mathbbm{1}\left\{\hat{\mu}_{k,s-1}\geq\max_{m\neq k}\hat{\mu}_{m,s-1}\right\}+\epsilon(s)/K$ for any $k\in[K]$, $s\ge1$, 
    we have 
    \begin{align}
        \sum_{s=1}^{t}P\left(\delta_s=k|F_{s-1}\right)&\leq\sum_{s=1}^{t}\mathbbm{1}\left\{\hat{\mu}_{k,s-1}\geq\max_{m\neq k}\hat{\mu}_{m,s-1}\right\}+\sum_{s=1}^{t}\epsilon(s)/K\\
        &\leq\sum_{s=1}^{t}\mathbbm{1}\left\{\hat{\mu}_{k,s-1}\geq\max_{m\neq k}\hat{\mu}_{m,s-1}\right\}+C/K+C\log(t)/K\label{eq:lemma3_1},
    \end{align}
    where \eqref{eq:lemma3_1} is due to the definition of $\epsilon(s)$ and $\sum_{s=C+1}^{t}1/s\leq\log(t)$ ($t\geq C+1$ and $C>1$ are required). Hence, we have 
    \begin{align}
        \mathbb{P}\left(N_{k,t}>\sqrt{t}\right)&\leq\mathbb{P}\left(\sum_{s=1}^{t}a_s>\sqrt{t}-\left(\sum_{s=1}^{t}\mathbbm{1}\left\{\hat{\mu}_{k,s-1}\geq\max_{m\neq k}\hat{\mu}_{m,s-1}\right\}+C/K+C\log(t)/K\right)\right)\\
        &=\mathbb{P}\left(\sum_{s=1}^{t}a_s>\sqrt{t}-E_t\right),\label{eq:lemma3_2}
    \end{align}
    where $E_t:=\sum_{s=1}^{t}\mathbbm{1}\left\{\hat{\mu}_{k,s-1}\geq\max_{m\neq k}\hat{\mu}_{m,s-1}\right\}+C/K+C\log(t)/K$ in \eqref{eq:lemma3_2}.
    Then, we have
    \begin{align}
        \mathbb{P}\left(N_{k,t}>\sqrt{t}\right)&\leq\mathbb{P}\left(\sum_{s=1}^{t}a_s>\sqrt{t}-E_t, E_t<\frac{1}{2}\sqrt{t}\right)+\mathbb{P}(E_t\geq \frac{1}{2}\sqrt{t})\label{eq:lemma3_3}\\
        &\leq\mathbb{P}\left(\sum_{s=1}^{t}a_s>\frac{1}{2}\sqrt{t}\right)+\frac{64K}{t^2}\left(\frac{1}{1-\exp\left\{-\frac{r^2}{512}\right\}}+2\right)\quad\forall\,t\geq\max\{T_2+1, T_{3,1}\}\label{eq:lemma3_4}\\
        &\leq\mathbb{P}\left(\sum_{s=1}^{t}a_s>\frac{1}{2}\sqrt{t}\right)+\frac{64K}{t^2}\left(3+\frac{512}{r^2}\right)\label{eq:lemma3_5},
    \end{align}
    where \eqref{eq:lemma3_3} is similar to \eqref{eq:lemma7_3}, \eqref{eq:lemma3_4} is from Lemma~\ref{le:numer_of_mis-orering_is_sublinear} with $T_2$ defined in Appendix~\ref{sec:proof_of_empirical_mean_error} and $T_{3,1}$ defined in Lemma~\ref{le:numer_of_mis-orering_is_sublinear}, and \eqref{eq:lemma3_5} is due to $1/(1-\exp(-x))\leq1+1/x$ for $x>0$.
    
    For the first term in\eqref{eq:lemma3_5}, following a similar argument in Appendix~\ref{a:proof_of_at_least_logt}, we have 
    \begin{align}
        \mathbb{P}\left(\sum_{s=1}^{t}a_s>\frac{1}{2}\sqrt{t}\right)&\leq 2\exp\left(-\frac{1}{2}\frac{t/4}{C\log(t)/2+C+\sqrt{t}/6}\right)\\
        &\leq2\exp\left(-\frac{1}{2}\frac{t/4}{\sqrt{t}/3}\right)\quad \forall\,t\geq 144C^2\log^2(12C)\\
        &=2\exp\left(-\frac{3}{8}\sqrt{t}\right).
    \end{align}
    Finally, by defining $T_3:=\max\left\{T_2+1, T_{3,1},144C^2\log^2(12C)\right\}$ and noticing that $\exp\left(-3\sqrt{t}/8\right)$ is faster decaying than $O(1/t^2)$, we can combine the above results and complete the proof.

\subsection{Proof of Theorem~\ref{thm:main_thm}}\label{a:proof_of_main_thm}
We combine our active selection strategy with the ideas in \cite{shekhar2025nonparametrictwosampletestingbetting} to prove Theorem~\ref{thm:main_thm}. Therefore, we will use the lemmas introduced in \ref{a:useful_lemma} to control the selection behavior. 
\begin{remark}
    For the analysis of the level-$\alpha$ property, all data samples are drawn from the null hypothesis $H_0$. For the analysis of the power-one property and the expected stopping time, all the lemmas in them are analyzed under the alternative hypothesis $H_1$. Hence, we may ignore the index $H_1$ in $\mathbb{P}_{H_1}$ and $\mathbb{E}_{H_1}$ for notational simplicity.
\end{remark}

\subsubsection{Proof of Level-$\alpha$ property}\label{a:level_alpha}
The goal is to show: for any $\left\{\mathcal{A}^{(k)}_{\text{pred}}\right\}_{k\in[K]}$, $\mathcal{A}_{\text{bet}}$, and selection policy defined in Definition~\ref{def:ASN2T}, given a predefined accuracy $\alpha\in(0,1)$, 
            \begin{align}
            \mathbb{P}_{H_0}(\tau^{(\text{ASN2T})}<\infty)\leq\alpha\quad\forall\,H_0.
            \end{align}

    We only need to show that the wealth process $\left\{W_t\right\}$ is a nonnegative supermartingale (or martingale) and apply Ville's inequality. $\forall\,H_0$, we have
    \begin{align}
        \mathbb{E}_{H_0}[W_{t}|F_{t-1}] &= W_{t-1} +W_{t-1}\mathbb{E}_{H_0}\left[\lambda_{t}v_{t}^{(\delta_t)}(g_{\delta_t,t})\,|\,F_{t-1}\right]\\
        &=W_{t-1}+W_{t-1}\sum_{k\in[K]}\mathbb{E}_{H_0}\left[\lambda_{t}(g_{k,t}(X_{k,1}(t))-g_{k,t}(X_{k,2}(t)))\,|\,F_{t-1},\delta_{t}=k\right]P(\delta_{t}=k\,|\,F_{t-1}) \label{eq:level_alpha_1}\\
        &=W_{t-1}\label{eq:level_alpha_2},
    \end{align}
    where \eqref{eq:level_alpha_1} is by the definition of $v^{(k)}_{t}(g)$ in \eqref{eq:v} and \eqref{eq:level_alpha_2} is due to the fact that $(X_{\delta_t,1}(t), X_{\delta_t,2}(t))\sim P_{\delta_t,1}\times P_{\delta_t,2}$, and $\lambda_{t}$ and $g_{k,t}$ are $F_{t-1}$ measurable, and under the null hypothesis, $X_{k,1}$ and $X_{k,2}$ have the same distribution for all $k\in[K]$. 
    
    Therefore, $\{W_t\}_{t\geq0}$ is a martingale. Then, since the range of $g_{k,t}$ for each $k\in[K]$ is in $[-1/2,1/2]$ and $\lambda_t\in[-1,1]$, $W_t\geq0$ for any $t\geq0$. That is, $\{W_t\}_{t\geq0}$ is a non-negative martingale since $1+\lambda_t v_{t}^{(\delta_t)}(g_{\delta_t,t})\geq0$.

    Consequently, we can apply Ville's inequality, leading to 
    \begin{align}
        \mathbb{P}_{H_0}(\tau^{(\text{ASN2T})}<\infty) = \mathbb{P}_{H_0}(\exists t\geq1:W_t\geq\frac{1}{\alpha})\leq \alpha,
    \end{align}
    where the last inequality is due to Ville's inequality.

\subsubsection{Proof of power-one property}\label{a:power_one_error_ctrl}
The goal is to show that, if almost surely,
            \begin{align}
                \limsup_{t\rightarrow\infty}\left(\frac{R^{(a^*)}_{t}}{t}-\frac{1}{t}\sum_{s=1}^{t}D^{(a^*)}_{\mathcal{G}}\mathbbm{1}\left\{\delta_s=a^*\right\}\right)<0,
            \end{align}
            then
            \begin{equation}
                \mathbb{P}_{H_1}(\tau^{(\text{ASN2T})}<\infty)=1\quad\forall H_1.
            \end{equation}
    Equivalently, let us prove $\mathbb{P}_{H_1}(\tau=\infty)=0$. By the property of the stopping time,
    \begin{align}
        \mathbb{P}_{H_1}(\tau^{(\text{ASN2T})}=\infty)\leq\mathbb{P}_{H_1}(\tau^{(\text{ASN2T})}>t)\quad\forall\,t\ge1,
    \end{align}
    we have $\mathbb{P}_{H_1}(\tau^{(\text{ASN2T})}=\infty)\leq\liminf_{t\rightarrow\infty}\mathbb{P}_{H_1}(\tau^{(\text{ASN2T})}>t)$.
    Then, we proceed to show $\liminf_{t\rightarrow\infty}\mathbb{P}_{H_1}(\tau>t)=0$.
    
    By the definition of $\tau^{(\text{ASN2T})}$, we have 
    \begin{align}
        \mathbb{P}_{H_1}(\tau^{(\text{ASN2T})}>t)\leq\mathbb{P}_{H_1}(W_{t}\leq\frac{1}{\alpha}).
    \end{align}
    Then, by Appendix~\ref{a:ons_betting}, we have  $W_t\geq\exp\left( \frac{t}{8} \left( \frac{1}{t} \sum_{s=1}^t v^{(\delta_s)}_{s}(g_{\delta_s,s}) \right)^2 - \log(t) \right)$. Plugging this into the above inequality, we have
    \begin{align}
        \mathbb{P}_{H_1}(W_{t}\leq\frac{1}{\alpha})&\leq\mathbb{P}_{H_1}\left(\exp\left( \frac{t}{8} \left( \frac{1}{t} \sum_{s=1}^t v^{(\delta_s)}_{s}(g_{\delta_s,s}) \right)^2 - \log(t) \right)\leq\frac{1}{\alpha}\right)\\
        &=\mathbb{P}_{H_1}\left(\frac{1}{t}\sum_{s=1}^{t}v^{(\delta_s)}_{s}(g_{\delta_s,s})\leq\sqrt{\frac{8(\log(1/\alpha)+\log(t))}{t}}\right)\label{eq:power_one_1},
    \end{align}
    where \eqref{eq:power_one_1} is rearrangement.
    
    Following this, by $b_t:=\sqrt{\frac{8(\log(1/\alpha)+\log(t))}{t}}$, we have,
    \begin{align}
        &\left\{\frac{1}{t}\sum_{s=1}^{t}v^{(\delta_s)}_{s}(g_{\delta_s,s})\leq b_t\right\}\\
        &=\left\{\frac{1}{t}\sum_{s=1}^{t}v^{(a^*)}_{s}(g_{a^*,s})\mathbbm{1}\{\delta_s=a^*\}+\frac{1}{t}\sum_{s=1}^{t}\sum_{k\neq a^*}v^{(k)}_{s}(g_{k,s})\mathbbm{1}\{\delta_s=k\}\leq b_t\right\}\\
        &\subseteq\left\{\frac{1}{t}\sum_{s=1}^{t}v^{(a^*)}_{s}(g_{a^*,s})\mathbbm{1}\left\{\delta_s=a^*\right\}\leq b_t+\sum_{k\neq a^*}\frac{N_{k,t}}{t},\cap_{k\neq a^*}\left\{N_{k,t}\leq\sqrt{t}\right\}\right\}\cup\left\{\cup_{k\neq a^*}\left\{N_{k,n}>\sqrt{n}\right\}\right\}\label{eq:power_one_2}\\
        &\subseteq\left\{\frac{1}{t}\sum_{s=1}^{t}v^{(a^*)}_{s}(g_{a^*,s})\mathbbm{1}\left\{\delta_s=a^*\right\}\leq b_t+\frac{K-1}{\sqrt{t}}\right\}\cup\left\{\cup_{k\neq a^*}\left\{N_{k,t}>\sqrt{t}\right\}\right\}\label{eq:power_one_3}\\
        &=\left\{\frac{1}{t}\sup_{g_{a^*}\in\mathcal{G}}\sum_{s=1}^{t}(g_{a^*}(X_{a^*,1}(s))-g_{a^*}(X_{a^*,2}(s)))\mathbbm{1}\left\{\delta_s=a^*\right\}\leq\frac{R^{(a^*)}_t}{t}+\frac{K-1}{\sqrt{t}}+b_t\right\}\cup\left\{\cup_{k\neq a^*}\left\{N_{k,t}>\sqrt{t}\right\}\right\}\label{eq:power_one_4}\\
        &\subseteq\left\{\frac{1}{t}\sum_{s=1}^{t}(g^*_{a^*}(X_{a^*,1}(s))-g^*_{a^*}(X_{a^*,2}(s)))\mathbbm{1}\left\{\delta_s=a^*\right\}\leq\frac{R^{(a^*)}_t}{t}+\frac{K-1}{\sqrt{t}}+b_t\right\}\cup\left\{\cup_{k\neq a^*}\left\{N_{k,t}>\sqrt{t}\right\}\right\}\label{eq:power_one_5}\\
        &=\left\{\frac{1}{t}\sum_{s=1}^{t}a_s\leq\frac{R^{(a^*)}_t}{t}-\frac{1}{t}\sum_{s=1}^{t}D^{(a^*)}_{\mathcal{G}}\mathbbm{1}\left\{\delta_s=a^*\right\}+\frac{K-1}{\sqrt{t}}+b_t\right\}\cup\left\{\cup_{k\neq a^*}\left\{N_{k,t}>\sqrt{t}\right\}\right\}\label{eq:power_one_6},
    \end{align}
    where \eqref{eq:power_one_2} is due to $|v^{k}_{s}(g)|\leq 1$ for all $k\in[K], g\in\mathcal{G}, s\ge1$, the definition of $N_{k,t}$, and the set operation in \eqref{eq:lemma7_3}, in \eqref{eq:power_one_3} we apply the fact that $\{N_{k,t}\leq \sqrt{t}\}$ for all $k\neq a^*$ for the first term, \eqref{eq:power_one_4} is due to the definition of the individual regret defined in \eqref{eq:individual_regret}, \eqref{eq:power_one_5} is due to the definition of supremum, and in \eqref{eq:power_one_6} we denote $a_s:=(g^*_{a^*}(X_{a^*,1}(s))-g^*_{a^*}(X_{a^*,2}(s))-D^{(a^*)}_{\mathcal{G}})\mathbbm{1}\left\{\delta_s=a^*\right\}$ for all $s\ge1$, which is a martingale difference.

    Therefore, plugging \eqref{eq:power_one_6} into \eqref{eq:power_one_1}, by union bound, we have 
    \begin{align}
        \mathbb{P}_{H_1}(W_{t}\leq\frac{1}{\alpha})\leq\mathbb{P}_{H_1}\left(\frac{1}{t}\sum_{s=1}^{t}a_s\leq\frac{R^{(a^*)}_t}{t}-\frac{1}{t}\sum_{s=1}^{t}D^{(a^*)}_{\mathcal{G}}\mathbbm{1}\left\{\delta_s=a^*\right\}+\frac{K-1}{\sqrt{t}}+b_t\right)+\sum_{k\neq a^*}\mathbb{P}_{H_1}(N_{k,t}>\sqrt{t})\label{eq:power_one_7}.
    \end{align}
    For the second term in \eqref{eq:power_one_7}, by Lemma~\ref{le:sub-opt_source_is_drawn_at_most_n0.5}, we get $\lim_{t\rightarrow\infty}\sum_{k\neq a^*}\mathbb{P}_{H_1}(N_{k,t}>\sqrt{t})$=0. Then, by Fatou's lemma, \eqref{eq:power_one_7} turns to $\liminf_{t\rightarrow\infty}\mathbb{P}_{H_1}(W_{t}\leq\frac{1}{\alpha})\leq\lim_{t\rightarrow\infty}\mathbb{E}_{H_1}[\mathbbm{1}\left\{A_t\right\}]$, where 
    \begin{equation}
        A_t:=\left\{F_t:\frac{1}{t}\sum_{s=1}^{t}a_s\leq\frac{R^{(a^*)}_t}{t}-\frac{1}{t}\sum_{s=1}^{t}D^{(a^*)}_{\mathcal{G}}\mathbbm{1}\left\{\delta_s=a^*\right\}+\frac{K-1}{\sqrt{t}}+b_t\right\}.
    \end{equation}
    
    We then show, almost surely, $\lim_{t\rightarrow\infty}\mathbbm{1}\left\{A_t\right\}=0$. Since $\sum_{t=1}^{\infty}\frac{\mathbb{E}[|a_t|^2]}{t^2}<\infty$ due to $\mathbb{E}_{H_1}[|a_t|^2]<\infty$, and $\{a_s\}_{s\ge0}$ is a martingale difference, by \cite[Theorem 2.15]{hall2014}, we have $\frac{1}{t}\sum_{s=1}^{t}a_s$ almost surely converges to $0$. Moreover, by the assumption \eqref{eq:condition_for_power-one},
    we have $\limsup_{n\rightarrow\infty}\text{LHS}\geq 0> \limsup_{n\rightarrow\infty}\text{RHS}$, where
    \begin{align}
        \text{LHS}:=\frac{1}{t}\sum_{s=1}^{t}a_s,\quad\text{RHS}:=\frac{R^{(a^*)}_t}{t}-\frac{1}{t}\sum_{s=1}^{t}D^{(a^*)}_{\mathcal{G}}\mathbbm{1}\left\{\delta_s=a^*\right\}+\frac{K-1}{\sqrt{t}}+b_t,
    \end{align}
    leading to $\mathbbm{1}\left\{A_t\right\}\overset{a.s.}{\rightarrow}0$. Therefore, by the bounded convergence theorem, we have $\lim_{t\rightarrow\infty}\mathbb{E}_{H_1}[\mathbbm{1}\left\{A_t\right\}]=0$. Hence, $\mathbb{P}_{H_1}(\tau=\infty)=0$ holds true. Proof completes.

\subsubsection{Proof of expected stopping time}\label{a:proof_of_expect_stop_time}
For the expected stopping time, we make some notations in advance, for all $k\in[K]$,
\begin{align}
    &\tilde{\sigma}^2_k:=\sup_{g_k\in\mathcal{G}}\mathbb{E}[(g_k(X_{k,1})-g_k(X_{k,2}))^2],\quad
    \gamma_k=\sup_{g_k\in\mathcal{G}}\text{Var}((g_k(X_{k,1})-g_k(X_{k,2}))^2).
\end{align}

The goal is to show that, if there exists $\left\{r_{t}\right\}_{t\ge1}$ and $E_{t,a^*}:=\left\{R^{(a^*)}_{t}/t\leq r_{t}\right\}$, then
                \begin{align}
                    \mathbb{E}_{H_1}[\tau^{(\text{ASN2T})}]\in O(t_0+\sum_{t\geq 1}\mathbb{P}(E^c_{t,a^*})+T_{\text{explore}})
                \end{align}
                where $T_{\text{explore}}$ is proportional to $K$ and $1/L$, and
                \begin{align}
                    t_0:=\inf\left\{t\ge1:D^{(a^*)}_{\mathcal{G}}\geq r_{t}+\sigma_{a^*}\sqrt{\frac{\log\frac{t}{\alpha}}{t}}+\frac{\log \frac{t}{\alpha}}{t}\right\},\quad\sigma_{k}:=\sqrt{\sup_{g_{k}}\text{Var}(g_{k}(X_{k,1})-g_{k}(X_{k,2}))}\quad\forall\,k\in[K].
                \end{align}
    The proof of the upper bound of the expected stopping time basically follows the steps in \cite[Appendix C.4]{shekhar2025nonparametrictwosampletestingbetting}. However, in our case, we need to control the action-taking behavior since the lemmas in Appendix~\ref{a:useful_lemma} will be used.

    Let us start from a common representation of the expected stopping time. In the following, we consider the alternative to be true and ignore the index $H_1$ in the notations of expectation and probability. Recall that
    \begin{align}
        \tau^{(\text{ASN2T})}:=\inf\left\{t\ge1:W_t\geq1/\alpha\right\},\quad\text{where}\ W_t=\Pi_{s=1}^{t}(1+\lambda_s v^{(\delta_s)}_{s}(g_{\delta_s,s})).\nonumber
    \end{align}
    We have
    \begin{align}
        \mathbb{E}[\tau^{(\text{ASN2T})}]&=\sum_{t=0}^{\infty}\mathbb{P}(\tau^{(\text{ASN2T})}>t) 
        \leq\sum_{t=0}^{\infty}\mathbb{P}(\log(W_t)<\log(1/\alpha)).
    \end{align}
    
    We proceed to show the upper bound of $\mathbb{P}(\log(W_t)<\log(1/\alpha))$. By the fact of ONS betting strategy \cite[Appendix C.4]{shekhar2025nonparametrictwosampletestingbetting},
    \begin{align}
        \log(W_t)\geq\sup_{\lambda\in[-1/2, 1/2]}\sum_{s=1}^{t}\log(1+\lambda v^{(\delta_s)}_{s}(g_{\delta_s,s}))-12\log(t)  ,  
    \end{align}
    we know 
    \begin{align}
        \mathbb{P}(\log(W_t)<\log(1/\alpha))&\leq\mathbb{P}\left(\sup_{\lambda\in[-1/2, 1/2]}\sum_{s=1}^{t}\log(1+\lambda v^{(\delta_s)}_{s}(g_{\delta_s,s})))-12\log(t)<\log(1/\alpha)\right)\\
        &=\mathbb{P}(A_t),
    \end{align}
    where
    \begin{align}
        A_t:=\left\{\sup_{\lambda\in[-1/2, 1/2]}\sum_{s=1}^{t}\log(1+\lambda v^{(\delta_s)}_{s}(g_{\delta_s,s})))-12\log(t)<\log(1/\alpha)\right\}.
    \end{align}
    Then, let us define $B_t:=\left\{\lvert\sum_{s=1}^{t}v^{(\delta_s)}_{s}(g_{\delta_s,s}))\rvert\leq\sum_{s=1}^{t}(v^{(\delta_s)}_{s}(g_{\delta_s,s}))^2\right\}$ and get 
    \begin{align}
        A_t=(A_t\cap B_t)\cup(A_t\cap B_t^c).
    \end{align}
    For $(A_t\cap B_t)$, we follow the first bullet point in \cite[Appendix C.4]{shekhar2025nonparametrictwosampletestingbetting} and get  
    \begin{align}
        (A_t\cap B_t)&\subseteq\left\{\frac{1}{4}\left(\sum_{s=1}^{t}v_s^{(\delta_s)}(g_{\delta_s,s})\right)^2/\sum_{s=1}^{t}(v_s^{(\delta_s)}(g_{\delta_s,s}))^2<12\log(\frac{t}{\alpha})\right\}\cap B_t\\
        &=\left\{\left(\frac{1}{t}\sum_{s=1}^{t}v_s^{(\delta_s)}(g_{\delta_s,s})\right)^2<\frac{48}{t}\sum_{s=1}^{t}(v_s^{(\delta_s)}(g_{\delta_s,s}))^2\cdot\frac{1}{t} \log(\frac{t}{\alpha})\right\}\cap B_t:=D_{t,1}.
    \end{align}
For $(A_t\cap B_t^c)$, we follow the second bullet point in \cite[Appendix C.4]{shekhar2025nonparametrictwosampletestingbetting}. Let us define two disjoint sets $C_{t,1}$ and $C_{t,2}$ such that $B_t^c:=C_{t,1}\cup C_{t,2}$ and 
\begin{align}
    C_{t,1}:=\left\{\sum_{s=1}^{t}v^{(\delta_s)}_{s}(g_{\delta_s,s})\geq\sum_{s=1}^{t}(v^{(\delta_s)}_{s}(g_{\delta_s,s}))^2\right\}\quad\text{and}\quad C_{t,2}:=\left\{\sum_{s=1}^{t}v^{(\delta_s)}_{s}(g_{\delta_s,s})\leq-\sum_{s=1}^{t}(v^{(\delta_s)}_{s}(g_{\delta_s,s}))^2\right\}.
\end{align}
Then, taking $\lambda=1/2$ in $A_t$ and $\log(1+\lambda v^{(\delta_s)}_{s}(g_{\delta_s,s}))\geq\lambda v^{(\delta_s)}_{s}(g_{\delta_s,s})-\lambda^2(v^{(\delta_s)}_{s}(g_{\delta_s,s}))^2$, we have 
\begin{align}
    (A_t\cap C_{t,1})&\subseteq\left\{\frac{1}{2}\sum_{s=1}^{t}v^{(\delta_s)}_{s}(g_{\delta_s,s})-\frac{1}{4}\sum_{t=1}^{n}(v^{(\delta_s)}_{s}(g_{\delta_s,s}))^2<12\log\frac{t}{\alpha}\right\}\cap C_{t,1}\\
    &\subseteq\left\{\frac{1}{4}\sum_{s=1}^{t}v^{(\delta_s)}_{s}(g_{\delta_s,s})<12\log\frac{t}{\alpha}\right\}\label{eq:exp_time_1}\\
    &=\left\{\frac{1}{t}\sum_{s=1}^{t}v^{(\delta_s)}_{s}(g_{\delta_s,s})<\frac{48}{t}\log\frac{t}{\alpha}\right\}:=D_{t,2},
\end{align}
where \eqref{eq:exp_time_1} is from the fact of $C_{t,1}$,
and
\begin{align}
    (A_t\cap C_{t,2})\subseteq C_{t,2}\subseteq\left\{\frac{1}{t}\sum_{s=1}^{t}v^{(\delta_s)}_{s}(g_{\delta_s,s})<0\right\}:=D_{t,3}.
\end{align}
Therefore, we have 
\begin{align}
    \mathbb{P}(\tau^{(\text{ASN2T})}>t)\leq\mathbb{P}(A_t)\leq\mathbb{P}(D_{t,1})+\mathbb{P}(D_{t,2})+\mathbb{P}(D_{t,3})\label{eq:exp_time_2},
\end{align}

Then, by \eqref{eq:exp_time_2}, the expected stopping time is upper-bounded by
\begin{align}
    \mathbb{E}[\tau^{(\text{ASN2T})}]\leq 1+\sum_{t\ge1}\mathbb{P}(D_{t,1})+\mathbb{P}(D_{t,2})+\mathbb{P}(D_{t,3}).
\end{align}

As a result, by Lemma~\ref{le:upper_bound_of_p(d_n2)}, Lemma~\ref{le:upper_bound_of_p(d_n3)}, and Lemma~\ref{le:upper_bound_of_p(d_n1)}, we have, 
\begin{align}
    \mathbb{E}[\tau^{\text{ASN2T}}]&\leq T+T_3+\sum_{t\geq T+T_3}3\mathbb{P}(E^c_{t})+3\mathbb{P}(E^c_{t,a^*})+3\mathbb{P}(G^c_{t,1})+\mathbb{P}(G^c_{t,2})
\end{align}
where $T:=\max\left\{t_1,t_2,t_3,t_4\right\}$, $t_1,t_2,t_3,t_4$ are defined in Lemma~\ref{le:upper_bound_of_p(d_n2)}, Lemma~\ref{le:upper_bound_of_p(d_n3)}, and Lemma~\ref{le:upper_bound_of_p(d_n1)}, and $T_3$ is defined in Appendix~\ref{a:sub-opt source is drawn at most sqrt{n}}. 

Consequently, by \cite[Lemma 3]{shekhar2025nonparametrictwosampletestingbetting}, we know $T\in O(t_0)$ where $t_0$ is defined in Theorem~\ref{thm:main_thm}, and with the fact that $T_{\text{explore}}\equiv T_3\in \tilde{O}(C^3)\asymp\tilde{O}\left(\frac{K^3}{L^6}\right)$ and $3\sum_{t\geq T_3}\left(\mathbb{P}(E^c_{t})+\mathbb{P}(G^c_{t,1})+\mathbb{P}(G^c_{t,2})\right)\in O(1)$, we get the desired result.

\subsection{Proof of Corollary~\ref{co:comparison}}\label{a:proof_of_comparison}
The goal is to show that
if the passive setting \cite{shekhar2025nonparametrictwosampletestingbetting} uses the sub-optimal data source and $\sigma_{a^*}^2\leq\max_{k\neq a^*}\sigma_{k}^2$,  
then
\begin{align}
        \mathbb{E}[\tau^{\text{ASN2T}}]\leq\mathbb{E}[\tau^{(p)}]+O\left(\frac{K^2}{r^2}(\log(1/\alpha))^{2/3}\right),
    \end{align}
    where $\tau^{(p)}$ denotes the stopping time in the general test proposed in \cite{shekhar2025nonparametrictwosampletestingbetting}.
    Asymptotically,  
    \begin{align}
        \lim_{\alpha\rightarrow0}\frac{\mathbb{E}[\tau^{\text{ASN2T}}]}{\log(1/\alpha)}\leq\lim_{\alpha\rightarrow0}\frac{\mathbb{E}[\tau^{(p)}]}{\log(1/\alpha)}.
    \end{align}

    The ground truth hypothesis is $H_1$ in this proof. Let us compare $t_1$, $t_2$, $t_3$ defined in the proof of the expected stopping time, Appendix~\ref{a:proof_of_expect_stop_time}, with those in the passive setting (we do not need to compare $t_4$ since they are the same). Assume in the passive setting, they use any data source $k$ where $k\neq a^*$ and assume $\sigma_{a^*}\leq \sigma_{k}$.
For $t_2^{(p)}$, representing the \emph{passive} version of $t_2$, from \cite{shekhar2025nonparametrictwosampletestingbetting}, we know
\begin{align}
    t_2^{(p)}=\inf\left\{t\ge1:\frac{D^{(k)}_{\mathcal{G}}}{2}\geq\frac{48}{t}\log\frac{t}{\alpha}+\sigma_{k}\sqrt{\frac{4\log(t)}{t}}+\frac{2\log(t)}{3t}+r_{t}\right\},
\end{align}
where $\sigma^2_k$ is the corresponding variance of that source.
Recall the definition of $t_2$ in this paper
\begin{align}
    t_2:=\inf\left\{t\ge1:\frac{D^{(a^*)}_{\mathcal{G}}}{2}\geq\frac{K-1}{\sqrt{t}}+\frac{48}{t}\log\frac{t}{\alpha}+\sigma_{a^*}\sqrt{\frac{4\log(t)}{t}}+\frac{2\log(t)}{3t}+r_t\right\}.
\end{align}
Therefore, we have $t_2= t_2^{(p)}+O(\frac{4(K-1)^2}{r^2})$.

Similarly, for $t_3^{(p)}$ defined as 
\begin{align}
    t_3^{(p)}=\inf\left\{t\ge1:\frac{D^{(k)}_{\mathcal{G}}}{2}\geq\sigma_{k}\sqrt{\frac{4\log(t)}{t}}+\frac{2\log(t)}{3t}+r_{t}\right\}. 
\end{align}
Recall the definition of $t_3$ in this paper
\begin{align}
    t_3:=\inf\left\{t\ge1:\frac{D^{(a^*)}_{\mathcal{G}}}{2}\geq\frac{K-1}{\sqrt{t}}+\sigma_{a^*}\sqrt{\frac{4\log(t)}{t}}+\frac{2\log(t)}{3t}+r_t\right\}.
\end{align}
Therefore, we also have $t_3= t_3^{(p)}+O(\frac{4(K-1)^2}{r^2})$.

Finally, for $t_1^{(p)}$ defined as 
\begin{align}
    t_1^{(p)}:=\inf\left\{t\ge1:\frac{D^{(k)}_{\mathcal{G}}}{2}\geq r_{t}+9\sigma_{k}\sqrt{\frac{2\log(t/\alpha)}{t}}+7\sqrt{2\tilde{\sigma}_k}\Big(\frac{\log\frac{t}{\alpha}}{t}\Big)^{3/4}+\frac{8\log\frac{t}{\alpha}}{t}\right\},
\end{align}
recall the definition of $t_1$ \eqref{eq:fine_bound_of_n_1},
\begin{align}
    t_1:=\inf\left\{t\ge1:\frac{D^{(a^*)}_{\mathcal{G}}}{2}\geq r_t+\frac{K-1}{\sqrt{t}}+7\frac{\sqrt{(K-1)\log\frac{t}{\alpha}}}{t^{3/4}}+9\sigma_{a^*}\sqrt{\frac{2\log(t/\alpha)}{t}}+7\sqrt{2\tilde{\sigma}_{a^*}}\Big(\frac{\log\frac{t}{\alpha}}{t}\Big)^{3/4}+\frac{8\log\frac{t}{\alpha}}{t}\right\},
\end{align}
we have $t_1= t_1^{(p)}+O\!\left(\frac{K^2}{r^2}\Big(\log\frac{1}{\alpha}\Big)^{2/3}\right)$ by a simple algebra.

Therefore, we have 
\begin{align}
    \max\left\{t_1,t_2,t_3\right\}\leq\max\left\{t_1^{(p)},t_2^{(p)},t_3^{(p)}\right\}+O\!\left(\frac{K^2}{r^2}\Big(\log\frac{1}{\alpha}\Big)^{2/3}\right).
\end{align}
Then, since $\sum_{t\geq t_0}\mathbb{P}(E_{t})+\mathbb{P}(E_{t,a^*})+\mathbb{P}(G_{t,1})+\mathbb{P}(G_{t,2})\in O(1)$ and $T_3\in o(\frac{K^2}{r^2}(\log\frac{1}{\alpha})^{2/3})$ as $\alpha$ is small, we have 
\begin{align}
    \mathbb{E}[\tau^{\text{ASN2T}}]\leq\mathbb{E}[\tau^{(p)}]+O\left(\frac{K^2}{r^2}(\log(1/\alpha))^{2/3}\right).
\end{align}
Asymptotically, we have 
\begin{align}
    \lim_{\alpha\rightarrow0}\frac{\mathbb{E}[\tau^{\text{ASN2T}}]}{\log(1/\alpha)}\leq\lim_{\alpha\rightarrow0}\frac{\mathbb{E}[\tau^{(p)}]}{\log(1/\alpha)}.
\end{align}

\subsection{Technical lemmas}\label{a:technical_lemmas}
\begin{lemma}[Ville's Inequality]\label{le:ville_ineq}
    Suppose $\left\{W_t:t\geq0\right\}$ is a nonnegative supermartingale process adapted to a filtration $\left\{F_t:t\geq0\right\}$. Then, we have, for any $a>0$, $\mathbb{P}(\exists t\geq1:W_t\geq a)\leq\frac{\mathbb{E}[W_0]}{a}$.
\end{lemma}

\begin{lemma}[Bernstein for martingale difference sequence \cite{dzhaparidze2001}]\label{le:bernstein}
    Let $\left\{X_t, F_t\right\}_{t\ge0}$ be a martingale difference sequence with $X_t\leq a$ almost surely and $\sum_{t=1}^{n}\mathbb{E}[X_t^2|F_{t-1}]\leq v$. Then, $\forall\,\delta>0$,
    \begin{align}
        \mathbb{P}\left(|\sum_{t=1}^{n}X_t|\geq\delta\right)\leq2\exp\left(-\frac{1}{2}\left(\frac{\delta^2}{v+\frac{a\delta}{3}}\right)\right)
    \end{align}
        
\end{lemma}

\begin{lemma}[Deviation of the empirical estimator for data source $a^*$]\label{le:deviation_of_esti_source1}
Assume $\exists M>0$ such that
\begin{align}
    \frac{R^{(a^*)}_t}{N_{a^*,t}}\leq \frac{M}{\sqrt{N_{a^*,t}}}\quad\text{almost surely}.\label{eq:lemma7_a}
\end{align}
Then,
    \begin{align}
        \mathbb{P}_{H_1}\left(D^{(a^*)}_{\mathcal{G}}-\hat{\mu}_{a^*,t}\geq\frac{r}{2}\right)\leq \frac{1}{1-\exp\left\{-\frac{r^2}{512}\right\}}\frac{1}{t^{5}}+\frac{2}{t^{5}}\quad\forall\,t>\max\left\{T_1, T_{2,1}\right\},
    \end{align}
 where $T_1$ are defined in Appendix~\ref{a:proof_of_at_least_logt} with $c_1\geq\frac{2560}{r^2}$ and $T_{2,1}:=\exp(16M^2/(r^2 c_1))$. 
\end{lemma}
\begin{proof}
    By the definition of the individual regret \eqref{eq:individual_regret} and the definition of $\hat{\mu}_{a^*,t}$, we have 
    \begin{align}
        D^{(a^*)}_{\mathcal{G}}-\hat{\mu}_{a^*,t}&=D^{(a^*)}_{\mathcal{G}}-\sup_{g_{a^*}\in\mathcal{G}}\frac{1}{N_{a^*,t}}\sum_{s=1}^{t}v^{(a^*)}_{s}(g_{a^*})\mathbbm{1}\left\{\delta_s=a^*\right\}+\frac{R^{(a^*)}_t}{N_{a^*,t}}\\
        &\leq D^{(a^*)}_{\mathcal{G}}- \frac{1}{N_{a^*,t}}\sum_{s=1}^{t}v^{(a^*)}_{s}(g^*_{a^*})\mathbbm{1}\left\{\delta_s=a^*\right\}+\frac{R^{(a^*)}_t}{N_{a^*,t}}\label{eq:lemma7_1}\\
        &\leq\frac{1}{N_{a^*,t}}\sum_{j=1}^{N_{a^*,t}}a_j+\frac{M}{\sqrt{N_{a^*,t}}}\label{eq:lemma7_2},
    \end{align}
    where $v^{(a^*)}_s(g^{*}_{a^*})$ and $g^{*}_{a^*}$  are relatively defined in \eqref{eq:v} and \eqref{eq:optimal_pred}, and \eqref{eq:lemma7_1} is from the definition of the supremum, and in \eqref{eq:lemma7_2}, we use the assumption \eqref{eq:lemma7_a} and the following definitions: 
    \begin{align}
        &a_j:=D^{(a^*)}_{\mathcal{G}}-(g^*_{a^*}(X_{a^*,1}(\tilde{T}_j))-g^*_{a^*}(X_{a^*,2}(\tilde{T}_j))),\quad\text{and}\label{eq:a_j_2}\\
        &\tilde{T}_j:=\inf\left\{s\ge1:\sum_{k=1}^{s}\mathbbm{1}\left\{\delta_{k}=a^*\right\}=j\right\}\quad\forall\,j\leq N_{a^*,t}.
    \end{align}
    Then, by \eqref{eq:lemma7_2}, we have
    \begin{align}
        \mathbb{P}_{H_1}\left(D^{(a^*)}_\mathcal{G}-\hat{\mu}_{a^*,t}\geq\frac{r}{2}\right)&\leq\mathbb{P}_{H_1}\left(\frac{1}{N_{a^*,t}}\sum_{j=1}^{N_{a^*,t}}a_j+\frac{M}{\sqrt{N_{a^*,t}}}\geq\frac{r}{2}\right)\\
        &\leq\mathbb{P}_{H_1}\left(\sum_{j=1}^{N_{a^*,t}}a_j\geq N_{a^*,t}\frac{r}{2}-M\sqrt{N_{a^*,t}},\, N_{a^*,t}\geq c_1\log(t)\right)+\mathbb{P}\left( N_{a^*,t}< c_1\log(t)\right)\label{eq:lemma7_3}\\
        &\leq \mathbb{P}_{H_1}\left(\sum_{j=1}^{N_{a^*,t}}a_j\geq N_{a^*,t}\frac{r}{2}-M\sqrt{N_{a^*,t}},\, N_{a^*,t}\geq c_1\log(t)\right) + \frac{2}{t^5}\quad\forall\,t\geq T_1,\label{eq:lemma7_4}
    \end{align}
    where \eqref{eq:lemma7_3} is due to $\{A\}=\{A\cap B\}\cup\{A\cap B^c\}\subseteq\{A\cap B\}\cup \{B^c\}$ and in \eqref{eq:lemma7_4}, we apply Lemma~\ref{le:at_least_logt} with $T_1$ and $c_1$ defined in Appendix~\ref{a:proof_of_at_least_logt}. We will set $c_1$ related to the sub-optimality gap $r$.
    
    For the first term, since $\left\{a_j\right\}_{j\ge1}$ is a martingale difference with respect to $\left\{F_{\tilde{T}_j}\right\}_{j\ge1}$, we apply Azuma's inequality. 
    \begin{align}
        \mathbb{P}_{H_1}\left(\sum_{j=1}^{N_{a^*,t}}a_j\geq N_{a^*,t}\frac{r}{2}-M\sqrt{N_{a^*,t}}, N_{a^*,t}\geq c_1\log(t)\right)&\leq\sum_{m=c_1\log(t)}^{\infty}\mathbb{P}_{H_1}\left(\sum_{j=1}^{m}a_j\geq m\frac{r}{2}-M\sqrt{m}\right)\label{eq:lemma7_5}\\
        &\leq\sum_{m=c_1\log(t)}^{\infty}\exp\left\{\frac{-m(\frac{r}{2}-M/\sqrt{m})^2}{32}\right\}\label{eq:lemma7_6}\\
        &\leq\sum_{m=c_1\log(t)}^{\infty}\exp\left\{\frac{-m(r^2/16)}{32}\right\}\label{eq:lemma7_7}\\
        &=\frac{1}{1-\exp\left\{-\frac{r^2}{512}\right\}}\exp\left\{-\frac{c_1r^2\log(t)}{512}\right\}\\
        &=\frac{1}{1-\exp\left\{-\frac{r^2}{512}\right\}}\frac{1}{t^{\frac{c_1r^2}{512}}},
    \end{align}
    where \eqref{eq:lemma7_5} is from union bound, \eqref{eq:lemma7_6} and \eqref{eq:lemma7_7} are due to Azuma's inequality with $|a_j-a_{j-1}|\leq4$ for all $j$, $\frac{mr}{2}-M\sqrt{m}>0$ and $(\frac{r}{2}-\frac{M}{\sqrt{m}})^2\geq r^2/16$ if we set $t>T_{2,1}:=\exp(16M^2/(r^2 c_1))$, i.e., $m\geq \frac{16M^2}{r^2}$. 

    Hence, if we set $c_1\geq\frac{2560}{r^2}$, then 
    \begin{align}
        \mathbb{P}_{H_1}\left(D^{(a^*)}_{\mathcal{G}}-\hat{\mu}_{a^*,t}\geq\frac{r}{2}\right)\leq \frac{1}{1-\exp\left\{-\frac{r^2}{512}\right\}}\frac{1}{t^{5}}+\frac{2}{t^{5}}\quad\forall\,t>\max\left\{T_1, T_{2,1}\right\},
    \end{align}
    Note that $M$ is bounded by the definition of $\mathcal{G}$ referred to \cite[Appendix A.4]{shekhar2025nonparametrictwosampletestingbetting}, which is finite.
\end{proof}

\begin{lemma}[Deviation of the empirical estimator for sub-optimal data source]\label{le:deviation_of_esti_source2}
For any $k\neq a^*$,
    \begin{align}
        \mathbb{P}_{H_1}\left(\hat{\mu}_{k,t}-D^{(k)}_\mathcal{G}\geq\frac{r}{2}\right)\leq \frac{1}{1-\exp(-r^2/(32))}\cdot\frac{1}{t^{5}}+\frac{2}{t^{5}}\quad\forall\,t>T_1
    \end{align}
 where 
 $T_{1}$ is defined in Appendix~\ref{a:proof_of_at_least_logt} with $c_1\geq\frac{160}{r^2}$.
\end{lemma}
\begin{proof}
    We will apply Azuma's inequality and Lemma~\ref{le:at_least_logt} in this proof. For any $k\neq a^*$, we have
    \begin{align}
        \hat{\mu}_{k,t}-D^{(k)}_{\mathcal{G}}&=\frac{1}{N_{k,t}}\sum_{s=1}^{t}v^{(k)}_s(g_{k,s})\mathbbm{1}\left\{\delta_s=k\right\}-D^{(k)}_{\mathcal{G}}\\
        &=\frac{1}{N_{k,t}}\sum_{s=1}^{t}\left(v^{(k)}_s(g_{k,s})\mathbbm{1}\left\{\delta_s=k\right\}-\mathbb{E}\left[v^{(k)}_s(g_{k,s})|F_{s-1}\right]\mathbbm{1}\left\{\delta_s=k\right\}\right)\\
        &\quad+\frac{1}{N_{k,t}}\sum_{s=1}^{t}\left(\mathbb{E}\left[v^{(k)}_s(g_{k,s})|F_{s-1}\right]\mathbbm{1}\left\{\delta_s=k\right\}-D^{(k)}_{\mathcal{G}}\mathbbm{1}\left\{\delta_s=k\right\}\right)\\
        &=\frac{1}{N_{k,t}}\sum_{j=1}^{N_{k,t}}\left(\tilde{v}^{(k)}_j(g_{k,j})-\mathbb{E}\left[\tilde{v}^{(k)}_j(g_{k,j})|F_{\tilde{T}_j-1}\right]\right)\\
        &\quad+\frac{1}{N_{k,t}}\sum_{j=1}^{N_{k,t}}\left(\mathbb{E}\left[\tilde{v}^{(k)}_j(g_{k,j})|F_{\tilde{T}_j-1}\right]-D^{(k)}_{\mathcal{G}}\right),\label{eq:lemma8_1}
    \end{align}
    where in \eqref{eq:lemma8_1}, we use $\tilde{v}^{(k)}_j(g_{k,j}):=g_{k,\tilde{T}_j}(X_{k,1}(\tilde{T}_j))-g_{k,\tilde{T}_j}(X_{k,2}(\tilde{T}_j))$ and
    $\tilde{T}_j:=\inf\left\{t\ge1:\sum_{s=1}^{t}\mathbbm{1}\left\{\delta_{s}=k\right\}=j\right\}$ for any $j\leq N_{k,t}$ and the expectation is taking over $H_1$.
    
    Then, since $g_{k,s}$ is $F_{s-1}$-measurable for all $s\ge1$ and $(X_{k,1}, X_{k,2})$ is independent of $F_{s-1}$ for any $s\ge1$, we know that 
    \begin{align}
        \mathbb{E}\left[\tilde{v}^{(k)}_j(g_{k,j})|F_{\tilde{T}_j-1}\right]\leq D^{(k)}_{\mathcal{G}}:=\sup_{g_k}\mathbb{E}[g_k(X_{k,1})-g_k(X_{k,2})]\quad\forall\,j\geq1, k\in[K].
    \end{align}
    Therefore, \eqref{eq:lemma8_1} turns to
    \begin{align}
        \hat{\mu}_{k,t}-D^{(k)}_{\mathcal{G}}\leq\frac{1}{N_{k,t}}\sum_{j=1}^{N_{k,t}}\left(\tilde{v}^{(k)}_j(g_{k,j})-\mathbb{E}\left[\tilde{v}^{(k)}_j(g_{k,j})|F_{\tilde{T}_j-1}\right]\right).
    \end{align}
    
    Following this, we denote $a_j:=\tilde{v}^{(k)}_j(g_{k,j})-\mathbb{E}\left[\tilde{v}^{(k)}_j(g_{k,j})|F_{\tilde{T}_j-1}\right]$
    and $G_j:=F_{\tilde{T}_j}$. It is easy to show that $\left\{a_j, G_j\right\}_{j\ge1}$ is a martingale difference by
    \begin{align}
        \mathbb{E}[a_j|G_{j-1}]&=\mathbb{E}[\tilde{v}^{(k)}_j(g_{k,j})|F_{T_{j-1}}]-\mathbb{E}\left[\mathbb{E}\left[\tilde{v}^{(k)}_j(g_{k,j})|F_{\tilde{T}_j-1}\right]|F_{T_{j-1}}\right]\\
        &=\mathbb{E}[\tilde{v}^{(k)}_j(g_{k,j})|F_{T_{j-1}}]-\mathbb{E}[\tilde{v}^{(k)}_j(g_{k,j})|F_{T_{j-1}}]\label{eq:lemma8_2}\\
        &=0,
    \end{align}
    where \eqref{eq:lemma8_2} is due to law of total expectation and $F_{T_{j-1}}\subseteq F_{T_{j}-1}$.
    
    Hence, plugging the definition of $a_j$, we have 
    \begin{align}
        \mathbb{P}_{H_1}\left(\hat{\mu}_{k,t}-D^{(k)}_\mathcal{G}\geq\frac{r}{2}\right)&\leq\mathbb{P}\left(\frac{1}{N_{k,t}}\sum_{j=1}^{N_{k,t}}a_j\geq\frac{r}{2}\right)\\
        &\leq\sum_{m=c_1\log(t)}^{\infty}\mathbb{P}\left(\sum_{j=1}^{m}a_j\geq m\frac{r}{2}\right)+\mathbb{P}(N_{k,t}<c_1\log(t))\label{eq:lemma8_3}\\
        &\leq\sum_{m=c_1\log(t)}^{\infty}\mathbb{P}\left(\sum_{j=1}^{m}a_j\geq m\frac{r}{2}\right)+\frac{2}{t^5}\quad\forall\,t\geq T_1\label{eq:deviation_for_source2}
    \end{align}
    where \eqref{eq:lemma8_3} and \eqref{eq:deviation_for_source2} are similar to  \eqref{eq:lemma7_3} and \eqref{eq:lemma7_4} and $T_1$ and $c_1$ are defined in Appendix~\ref{a:proof_of_at_least_logt}. We will define $c_1$ related to $r$.

    For the first term, we apply Azuma's inequality, which is basically the same as \ref{eq:lemma7_6}. 
    \begin{align}
        \sum_{m=c_1\log(t)}^{\infty}\mathbb{P}\left(\sum_{j=1}^{m}a_j\geq m\frac{r}{2}\right)&\leq\sum_{m=c_1\log(t)}^{\infty}\exp\left(-\frac{mr^2}{32}\right)\\
        &=\frac{1}{1-\exp(-r^2/(32))}\exp\left(-\frac{c_1\log(t)r^2}{32}\right)\\
        &=\frac{1}{1-\exp(-r^2/(32))}\cdot\frac{1}{t^{\frac{c_1r^2}{32}}}\\
        &\leq\frac{1}{1-\exp(-r^2/(32))}\cdot\frac{1}{t^{5}},
    \end{align}
    where the last inequality is true if we set $c_1\geq\frac{160}{r^2}$. Then, we complete the proof.
\end{proof}

\begin{lemma}[The number of mis-ordering is sublinear]\label{le:numer_of_mis-orering_is_sublinear}
For any $k\neq a^*$, $\forall\,t\geq \max\left\{T_{3,1}, T_2+1\right\}$,
\begin{align}
    \mathbb{P}\left(\sum_{s=1}^{t}\mathbbm{1}\left\{\hat{\mu}_{k,s-1}\geq\max_{m\neq k}\hat{\mu}_{m,s-1}\right\}\geq \frac{1}{2}\sqrt{t}-C/K-C\log(t)/K\right)\leq\frac{64K}{t^2}\left(\frac{1}{1-\exp\left\{-\frac{r^2}{512}\right\}}+2\right),
\end{align}
where $T_{3,1}:=\frac{64C^2}{K^2}\left(\log\frac{8C}{K}\right)^2$ and $T_2$ is defined in Appendix~\ref{sec:proof_of_empirical_mean_error}.
\end{lemma}
\begin{proof}
    We define $\tilde{\tau}:=\inf\left\{t\geq1:\hat{\mu}_{k,s-1}<\max_{m\neq k}\hat{\mu}_{m,s-1},\,\forall\,s\geq t\right\}$ and a high probability event $\left\{\tilde{\tau}\leq \frac{1}{4}\sqrt{t}\right\}$. Then, we have 
    \begin{align}
        &\mathbb{P}\left(\sum_{s=1}^{t}\mathbbm{1}\left\{\hat{\mu}_{k,s-1}\geq\max_{m\neq k}\hat{\mu}_{m,s-1}\right\}\geq\frac{1}{2}\sqrt{t}-C/K-C\log(t)/K\right)\\
        &\leq\mathbb{P}\left(\sum_{s=1}^{t}\mathbbm{1}\left\{\hat{\mu}_{k,s-1}\geq\max_{m\neq k}\hat{\mu}_{m,s-1}\right\}\geq\frac{1}{2}\sqrt{t}-C/K-C\log(t)/K, \tilde{\tau}\leq\frac{1}{4}\sqrt{t}\right)+\mathbb{P}\left(\tilde{\tau}>\frac{1}{4}\sqrt{t}\right)\label{eq:lemma9_1}.
    \end{align}
    where \eqref{eq:lemma9_1} is similar to \eqref{eq:lemma7_3}.
    For the first term, under the event $\left\{\tilde{\tau}\leq\frac{1}{4}\sqrt{t}\right\}$, we have 
    \begin{align}
        \sum_{s=1}^{t}\mathbbm{1}\left\{\hat{\mu}_{k,s-1}\geq\max_{m\neq k}\hat{\mu}_{m,s-1}\right\}=\sum_{s=1}^{\tilde{\tau}}\mathbbm{1}\left\{\hat{\mu}_{k,s-1}\geq\max_{m\neq k}\hat{\mu}_{m,s-1}\right\}\leq\frac{1}{4}\sqrt{t}
    \end{align}
    since after $s\ge\tilde{\tau}+1$, $\mathbbm{1}\left\{\hat{\mu}_{k,s-1}\geq\max_{m\neq k}\hat{\mu}_{m,s-1}\right\}=0$. Hence, for any $t\geq T_{3,1}:=\frac{64C^2}{K^2}\left(\log\frac{8C}{K}\right)^2$, we have
    \begin{align}
        \frac{1}{4}\sqrt{t}<\frac{1}{2}\sqrt{t}-C/K-C\log(t)/K,
    \end{align}
    leading to zero for the first term in \eqref{eq:lemma9_1}.
    
    For the second term in \eqref{eq:lemma9_1}, we have 
    \begin{align}
        \mathbb{P}\left(\tilde{\tau}>\frac{1}{4}\sqrt{t}\right)=\sum_{m=\frac{1}{4}\sqrt{t}+1}^{\infty}\mathbb{P}(\tilde{\tau}=m)\leq\sum_{m=\frac{1}{4}\sqrt{t}}^{\infty}\mathbb{P}(\hat{\mu}_{k,m-1}\geq\max_{a\neq k}\hat{\mu}_{a,m-1}).
    \end{align}
    where the last inequality is due to
    \begin{align}
        \left\{\tilde{\tau}=m\right\}=\left\{\exists s< m-2, \hat{\mu}_{k,s-1}\geq\max_{m\neq k}\hat{\mu}_{m,s-1}\right\}\cap\left\{\hat{\mu}_{k,m-2}\geq\max_{a\neq k}\hat{\mu}_{a,m-2}\right\}\cap\left\{\forall\,s\geq m, \hat{\mu}_{k,s-1}<\max_{m\neq k}\hat{\mu}_{m,s-1}\right\},
    \end{align}
    and $P(A\cap B\cap C)\leq P(B)$.
    Then, by Appendix~\ref{sec:proof_of_empirical_mean_error}, we have, for all $t\geq T_2+1$,
    \begin{align}
        \mathbb{P}\left(\tilde{\tau}>\frac{1}{4}\sqrt{t}\right)&\leq\sum_{m=\frac{1}{4}\sqrt{t}}^{\infty}\frac{K}{1-\exp\left\{-\frac{r^2}{512}\right\}}\frac{1}{m^{5}}+\frac{2K}{m^{5}}\\
        &\leq\frac{64K}{t^2}\left(\frac{1}{1-\exp\left\{-\frac{r^2}{512}\right\}}+2\right)\label{eq:lemma9_2}.
    \end{align}
    where \eqref{eq:lemma9_2} is due to $\frac{1}{4}\sqrt{t}\gg1$. Proof complete.
    
\end{proof}

\begin{lemma}[Upper bound of $\mathbb{P}(D_{t,2})$]\label{le:upper_bound_of_p(d_n2)}
For all $t\geq \max\left\{68, t_2\right\}$, where 
\begin{align}
    t_2:=\inf\left\{t\ge 1:\frac{D^{(a^*)}_\mathcal{G}}{2}\geq\frac{48}{t}\log\frac{t}{\alpha}+\frac{K-1}{\sqrt{t}}+r_t+\sigma_{a^*}\sqrt{\frac{4\log(t)}{t}}+\frac{2\log(t)}{3t}\right\},
\end{align}
we have
\begin{align}
    \mathbb{P}(D_{t,2})\leq\mathbb{P}(E^c_{t})+\mathbb{P}(E^c_{t,a^*})+\mathbb{P}(G_{t,1}^c),
\end{align}
where $E_{t}:=\left\{\cap_{k\neq a^*}\left\{N_{k,t}\leq\sqrt{t}\right\}\right\}$, $E_{t,a^*}$ is defined in Theorem~\ref{thm:main_thm}, and $G_{t,1}$ is defined in Lemma~\ref{le:technical_lemma_for_stopping_time}.
\end{lemma}
\begin{proof}
    For notational simplicity, we define $v_{k,s}:=v^{(k)}_{s}(g_{k,s})$ for any $k\in[K]$, $s\ge1$, $g_{k,s}\in\mathcal{G}$ in this proof. Recall $D_{t,2}:=\left\{\frac{1}{t}\sum_{s=1}^{t}v_{\delta_s,s}<\frac{48}{t}\log\frac{t}{\alpha}\right\}$. Define the high probability events $E_{t}:=\left\{\cap_{k\neq a^*}\left\{N_{k,t}\leq\sqrt{t}\right\}\right\}$. We have $\mathbb{P}(D_{t,2})\leq\mathbb{P}(D_{t,2}\cap E_{t}\cap E_{t,a^*})+\mathbb{P}(E_{t}^c)+\mathbb{P}(E_{t,a^*}^c)$ by a simple set operation. Then, we only need to deal with the first term since the second term is summable by Lemma~\ref{le:sub-opt_source_is_drawn_at_most_n0.5} and the third term is also summable by the assumption in Theorem~\ref{thm:main_thm}. 
    
    For the first term $\mathbb{P}(D_{t,2}\cap E_{t}\cap E_{t,a^*})$, we do the similar steps in \eqref{eq:power_one_2} to \eqref{eq:power_one_5}, by $|v_{k,t}|\leq1$ for all $k\in[K]$,
    \begin{align}
        (D_{t,2}\cap E_{t}\cap E_{t,a^*})&=\left\{\frac{1}{t}\sum_{s=1}^{t}v_{a^*,s}\mathbbm{1}\left\{\delta_s=a^*\right\}+\frac{1}{t}\sum_{k\neq a^*}\sum_{s=1}^{t}v_{k,s}\mathbbm{1}\left\{\delta_s=k\right\}<\frac{48}{t}\log\frac{t}{\alpha}\right\}\cap E_{t}\cap E_{t,a^*}\\
        &\subseteq\left\{\frac{1}{t}\sum_{s=1}^{t}v_{a^*,s}\mathbbm{1}\left\{\delta_s=a^*\right\}-\sum_{k\neq a^*}\frac{N_{k,t}}{t}<\frac{48}{t}\log\frac{t}{\alpha}\right\}\cap E_{t}\cap E_{t,a^*}\\
        &\subseteq\left\{\frac{1}{t}\sum_{s=1}^{t}v_{a^*,s}\mathbbm{1}\left\{\delta_s=a^*\right\}<\frac{48}{t}\log\frac{t}{\alpha}+\frac{K-1}{\sqrt{t}}\right\}\cap E_{t,a^*}\label{eq:bound_of_d_n2_1}\\
        &=\left\{\sup_{g_{a^*}}\frac{1}{t}\sum_{s=1}^{t}g_{a^*}(X_{a^*,1}(s))-g_{a^*}(X_{a^*,2}(s))\mathbbm{1}\left\{\delta_s=a^*\right\}<\frac{48}{t}\log\frac{t}{\alpha}+\frac{K-1}{\sqrt{t}}+r_{t}\right\}\\
        &\subseteq\left\{\frac{1}{t}\sum_{s=1}^{t}g^*_{a^*}(X_{a^*,1}(s))-g^*_{a^*}(X_{a^*,2}(s))\mathbbm{1}\left\{\delta_s=a^*\right\}<\frac{48}{t}\log\frac{t}{\alpha}+\frac{K-1}{\sqrt{t}}+r_t\right\},
    \end{align}

    Then, by Lemma~\ref{le:technical_lemma_for_stopping_time}, with
    \begin{align}
        G_{t,1}:=\left\{\frac{1}{t}\sum_{s=1}^{t}g^*_{a^*}(X_{a^*,1}(s))-g^*_{a^*}(X_{a^*,2}(s))\mathbbm{1}\left\{\delta_s=a^*\right\}>D^{(a^*)}_{\mathcal{G}}-\tilde{\sigma}_{a^*}\sqrt{\frac{4\log(t)}{t}}-\frac{2\log(t)}{3t}\right\},
    \end{align}
    we know $\mathbb{P}(G^c_{t,1})\leq\frac{1}{t^2}$ and further get
    \begin{align}
    (D_{t,2}\cap E_{t}\cap E_{t,a^*})
    &\subseteq
    \left(\left\{\frac{1}{t}\sum_{s=1}^{t}g^*_{a^*}(X_{a^*,1}(s))-g^*_{a^*}(X_{a^*,2}(s))\mathbbm{1}\left\{\delta_s=a^*\right\}
    <\frac{48}{t}\log\frac{t}{\alpha}+\frac{K-1}{\sqrt{t}}+r_t\right\}\cap G_{t,1}\right)\cup G_{t,1}^c\\
    &\subseteq\left\{D^{(a^*)}_{\mathcal{G}}<\frac{48}{t}\log\frac{t}{\alpha}+\frac{K-1}{\sqrt{t}}+r_t\tilde{\sigma}_{a^*}\sqrt{\frac{4\log(t)}{t}}+\frac{2\log(t)}{3t}\right\}\cup G_{t,1}^c.
    \end{align}
    Consequently, if
    \[
    t\geq \inf\left\{t\ge 1:\;
    D^{(a^*)}_\mathcal{G}\geq
    \frac{48}{t}\log\frac{t}{\alpha}+\frac{K-1}{\sqrt{t}}+r_t
    +\tilde{\sigma}_{a^*}\sqrt{\frac{4\log(t)}{t}}+\frac{2\log(t)}{3t}\right\},
    \]
    we get
    \begin{align}
    \mathbb{P}(D_{t,2})\leq\mathbb{P}(E^c_{t})+\mathbb{P}(E^c_{t,2})+\mathbb{P}(G_{t,1}^c).
    \end{align}
    In addition, we can further apply Lemma~\ref{le:basic_ineq} and follow the same argument (simple algebra) in \cite[Appendix C.4]{shekhar2025nonparametrictwosampletestingbetting}, to get if $t\geq \max\left\{68, t_2\right\}$, where 
    \begin{align}
    t_2:=\inf\left\{t\ge 1:\;
    \frac{D^{(a^*)}_\mathcal{G}}{2}\geq
    \frac{48}{t}\log\frac{t}{\alpha}+\frac{K-1}{\sqrt{t}}+r_t
    +\sigma_{a^*}\sqrt{\frac{4\log(t)}{t}}+\frac{2\log(t)}{3t}\right\},    
    \end{align}
    then $\mathbb{P}(D_{t,2})\leq\mathbb{P}(E^c_{t})+\mathbb{P}(E^c_{t,2})+\mathbb{P}(G_{t,1}^c)$ also holds true.
\end{proof}

\begin{lemma}[Upper bound of $\mathbb{P}(D_{t,3})$]\label{le:upper_bound_of_p(d_n3)}
For all $t\geq\max\left\{68, t_3\right\}$, where 
\begin{align}
    t_3:=\inf\left\{t\ge1:\frac{D^{(a^*)}_{\mathcal{G}}}{2}\geq\sigma_{a^*}\sqrt{\frac{4\log(t)}{t}}+\frac{2\log(t)}{3t}+r_t+\frac{K-1}{\sqrt{t}}\right\},
\end{align}
we have
\begin{align}
    \mathbb{P}(D_{t,3})\leq\mathbb{P}(E^c_{t})+\mathbb{P}(E^c_{t,a^*})+\mathbb{P}(G_{t,1}^c),
\end{align}
where $E_{t}$ is defined in Lemma~\ref{le:upper_bound_of_p(d_n2)}, $E_{t,a^*}$ is defined in Theorem~\ref{thm:main_thm}, and $G_{t,1}$ is defined in Lemma~\ref{le:technical_lemma_for_stopping_time}.
\end{lemma}

\begin{proof}
    The proof of this bound is similar to Lemma~\ref{le:upper_bound_of_p(d_n2)}. We have 
    \begin{align}
        \mathbb{P}(D_{t,3})\leq\mathbb{P}(D_{t,3}\cap E_{t}\cap E_{t,a^*}\cap G_{t,1})+\mathbb{P}(E^c_{t})+\mathbb{P}(E^c_{t,a^*})+\mathbb{P}(G_{t,1}^c).
    \end{align}
    Then, by a similar argument in Lemma~\ref{le:upper_bound_of_p(d_n2)},
    \begin{align}
       \left\{D_{t,3}\cap E_{t}\cap E_{t,a^*}\cap G_{t,1}\right\}\subseteq\left\{D^{(a^*)}_{\mathcal{G}}<\tilde{\sigma}_{a^*}\sqrt{\frac{4\log(t)}{t}}+\frac{2\log(t)}{3t}+r_t+\frac{K-1}{\sqrt{t}}\right\}.
    \end{align}
    Therefore, we also use Lemma~\ref{le:basic_ineq} and follow similar steps (simple algebra) in \cite[Appendix C.4]{shekhar2025nonparametrictwosampletestingbetting} to get, for all $t\geq \max\left\{68, t_{3}\right\}$, where
    \begin{align}
        t_3:=\inf\left\{t\ge1:\frac{D^{(a^*)}_{\mathcal{G}}}{2}\geq\sigma_{a^*}\sqrt{\frac{4\log(t)}{t}}+\frac{2\log(t)}{3t}+r_t+\frac{K-1}{\sqrt{t}}\right\},
    \end{align}
    $\mathbb{P}(D_{t,3})\leq\mathbb{P}(E^c_{t})+\mathbb{P}(E^c_{t,a^*})+\mathbb{P}(G_{t,1}^c)$.  
\end{proof}

 \begin{remark}
        The major difference in Lemma~\ref{le:upper_bound_of_p(d_n2)} and Lemma~\ref{le:upper_bound_of_p(d_n3)} compared to \cite{shekhar2025nonparametrictwosampletestingbetting} is that we introduce the best data source in the definition of $t_2$. However, we need to pay an exploration cost, $\frac{K-1}{\sqrt{t}}$, but it is a second-order term.
\end{remark}

\begin{lemma}[Upper bound of $\mathbb{P}(D_{t,1})$]\label{le:upper_bound_of_p(d_n1)}
    For all $t\geq \max\left\{t_1, t_4\right\}$, where 
\begin{align}
    &t_1:=\inf\left\{t\ge1:\frac{D^{(a^*)}_{\mathcal{G}}}{2} \ge r_t + 9\sigma_{a^*} \sqrt{\frac{2\log(t/\alpha)}{t}} + 7\sqrt{2\tilde{\sigma}_{a^*}} \left( \frac{\log(t/\alpha)}{t} \right)^{3/4} + \frac{8\log(t/\alpha)}{t}+\frac{K-1}{\sqrt{t}}+7\frac{\sqrt{(K-1)\log(t/\alpha)}}{t^{3/4}}\right\},\\
    &t_{4}:=\inf\left\{t \ge 1 : t/\log(t/\alpha)\ge 392\right\},
\end{align}
the probability of $\mathbb{P}(D_{t,1})$ satisfies
\begin{align}
    \mathbb{P}(D_{t,1})\leq\mathbb{P}(E_{t}^c)+\mathbb{P}(E_{t,a^*}^c)+\mathbb{P}(G^c_{t,2})+\mathbb{P}(G^c_{t,1}),
\end{align}
where $E_{t}$ is defined in Lemma~\ref{le:upper_bound_of_p(d_n2)},  $E_{t,a^*}$ is defined in Theorem~\ref{thm:main_thm}, and $G_{t,1}$ and $G_{t,2}$ are defined in Lemma~\ref{le:technical_lemma_for_stopping_time}.
\end{lemma}

\begin{proof}
    For notational simplicity, we define $v_{k,s}:=v^{(k)}_{s}(g_{k,s})$ for any $k\in[K]$, $s\ge1$, $g_{k,s}\in\mathcal{G}$ in this proof. Recall that
    \[
    D_{t,1}:=\left\{\left(\frac{1}{t}\sum_{s=1}^{t}v_{\delta_s,s}\right)^2<\frac{48}{t}\sum_{s=1}^{t}v_{\delta_s,s}^2\cdot\frac{1}{t} \log\Big(\frac{t}{\alpha}\Big)\right\}\cap B_t
    \quad\text{and}\quad
    B_t:=\left\{\Bigl\lvert\sum_{s=1}^{t}v_{\delta_s,s}\Bigr\rvert\leq\sum_{s=1}^{t}v^2_{\delta_s,s}\right\}.
    \]
    By a similar set operation in Lemma~\ref{le:upper_bound_of_p(d_n2)}, we have
    \begin{align}
        \mathbb{P}(D_{t,1})\leq\mathbb{P}(D_{t,1}\cap E_{t}\cap E_{t,a^*})+\mathbb{P}(E_{t}^c)+\mathbb{P}(E_{t,a^*}^c).
    \end{align}
    Moreover,
    \begin{align}
        &\left\{D_{t,1}\cap E_{t}\cap E_{t,a^*}\right\}\\
        &\subseteq\left\{\left(\frac{1}{t}\sum_{s=1}^{t}v_{\delta_s,s}\right)^2<\frac{48}{t}\sum_{s=1}^{t}v_{\delta_s,s}^2\cdot\frac{1}{t} \log\Big(\frac{t}{\alpha}\Big)\right\}\cap E_{t}\cap E_{t,a^*}\\
        &=\left\{|\frac{1}{t}\sum_{s=1}^{t}v_{\delta_s,s}|<7\sqrt{\frac{1}{t}\sum_{s=1}^{t}v^2_{\delta_s,s}}\cdot\sqrt{\frac{1}{t}\log\frac{t}{\alpha}}\right\}\cap E_{t}\cap E_{t,a^*}\label{eq:lemma12_1}\\
        &\subseteq\left\{\frac{1}{t}\sum_{s=1}^{t}v_{\delta_s,s}<7\sqrt{\frac{1}{t}\sum_{s=1}^{t}v^2_{\delta_s,s}}\cdot\sqrt{\frac{1}{t}\log\frac{t}{\alpha}}\right\}\cap E_{t}\cap E_{t,a^*}\label{eq:lemma12_2}\\
        &\subseteq\left\{\frac{1}{t}\sum_{s=1}^{t}v_{a^*,s}\mathbbm{1}\left\{\delta_s=a^*\right\}<\frac{K-1}{\sqrt{t}}+7\sqrt{\frac{1}{t}\sum_{s=1}^{t}v^2_{\delta_s,s}}\cdot\sqrt{\frac{1}{t}\log\frac{t}{\alpha}}\right\}\cap E_{t}\cap E_{t,a^*}\label{eq:bound_of_d_n1_1}\\
        &\subseteq\left\{\frac{1}{t}\sum_{s=1}^{t}v_{a^*,s}\mathbbm{1}\left\{\delta_s=a^*\right\}<\frac{K-1}{\sqrt{t}}+7\sqrt{\frac{1}{t}\sum_{s=1}^{t}v^2_{a^*,s}\mathbbm{1}\left\{\delta_s=a^*\right\}}\cdot\sqrt{\frac{\log\frac{t}{\alpha}}{t}}+7\frac{\sqrt{(K-1)\log\frac{t}{\alpha}}}{t^{3/4}}\right\}\cap E_{t,a^*}\label{eq:bound_of_d_n1_2},
    \end{align}
    where \eqref{eq:lemma12_2} is from $\{|A|\leq l\}\subseteq\{A\leq l\}$, \eqref{eq:bound_of_d_n1_1} follows the argument in \eqref{eq:power_one_3}, and \eqref{eq:bound_of_d_n1_2} is due to $\sqrt{ab}\leq\sqrt{a}+\sqrt{b}$ for non-negative $a,b$ and a similar argument in \eqref{eq:power_one_3}. Then, by a similar argument in Lemma~\ref{le:upper_bound_of_p(d_n2)} to introduce $\{r_t\}_t$, we also have
    \begin{align}
        &\left\{D_{t,1}\cap E_{t}\cap E_{t,a^*}\right\}\\
        &\subseteq\left\{\frac{1}{t}\sum_{s=1}^{t}v^*_{a^*}(s)\mathbbm{1}\left\{\delta_s=a^*\right\}<r_t+\frac{K-1}{\sqrt{t}}+7\sqrt{\frac{1}{t}\sum_{s=1}^{t}v^2_{a^*,s}\mathbbm{1}\left\{\delta_s=a^*\right\}}\cdot\sqrt{\frac{\log\frac{t}{\alpha}}{t}}+7\frac{\sqrt{(K-1)\log\frac{t}{\alpha}}}{t^{3/4}}\right\},
    \end{align}
    where $v^*_{a^*}(s):=g^*_{a^*}(X_{a^*,1}(s))-g^*_{a^*}(X_{a^*,2}(s))\mathbbm{1}\left\{\delta_s=a^*\right\}$.
    
    Following this, from Lemma~\ref{le:technical_lemma_for_stopping_time}, we define
    \[
    G_{t,2}:=\left\{\frac{1}{t}\sum_{s=1}^{t}v^2_{a^*,s}\mathbbm{1}\left\{\delta_s=a^*\right\}\leq\tilde{\sigma}^2_{a^*}+\gamma_{a^*}\sqrt{\frac{4\log(t)}{t}}+\frac{2\log(t)}{3t}\right\}
    \]
    and we can further do a set operation $\left\{D_{t,1}\cap E_{t}\cap E_{t,a^*}\right\}\subseteq\left\{D_{t,1}\cap E_{t}\cap E_{t,a^*}\cap G_{t,2}\right\}\cup G_{t,2}^c$,
    leading to
    \begin{align}
        \sqrt{\frac{1}{t}\sum_{s=1}^{t}v^2_{a^*,s}\mathbbm{1}\left\{\delta_s=a^*\right\}}&\leq\sqrt{\tilde{\sigma}^2_{a^*}+\gamma_{a^*}\sqrt{\frac{4\log(t)}{t}}+\frac{2\log(t)}{3t}}\\
        &\leq\tilde{\sigma}_{a^*}+\sqrt{\gamma_{a^*}}\left(\frac{4\log(t)}{t}\right)^{1/4}+\sqrt{\frac{2\log(t)}{3t}}\label{eq:lemma12_4}\\
        &\leq\tilde{\sigma}_{a^*}+\sqrt{2\gamma_{a^*}}\left(\frac{\log\frac{t}{\alpha}}{t}\right)^{1/4}+\sqrt{\frac{\log \frac{t}{\alpha}}{t}}\\
        &\leq\sigma_{a^*}+D^{(a^*)}_{\mathcal{G}}+\sqrt{2\gamma_{a^*}}\left(\frac{\log\frac{t}{\alpha}}{t}\right)^{1/4}+\sqrt{\frac{\log \frac{t}{\alpha}}{t}}\label{eq:lemma12_5},
    \end{align}
    where \eqref{eq:lemma12_4} holds true by the subadditivity of the square, \eqref{eq:lemma12_5} is due to Lemma~\ref{le:basic_ineq}. 
    
    Then, to get a similar result in Lemma~\ref{le:upper_bound_of_p(d_n2)} and Lemma~\ref{le:upper_bound_of_p(d_n1)}, we further do
    \begin{align}
        \left\{D_{t,1}\cap E_{t}\cap E_{t,a^*}\cap G_{t,2}\right\}\subseteq\{D_{t,1}\cap E_{t}\cap E_{t,a^*}\cap G_{t,2}\cap G_{t,1}\}\cup G_{t,1}^c,
    \end{align}
    hence, we
    can follow the same steps in \cite[Appendix C.4]{shekhar2025nonparametrictwosampletestingbetting} to get, for all
    \[
    t\geq t_{4}:= \inf\left\{m \ge 1 : 7\sqrt{2\log(m/\alpha)}/m \le 1/2\right\}=\inf\left\{m \ge 1 : m/\log(m/\alpha)\ge 392\right\},
    \]
    \begin{align}
        &\left\{D_{t,1}\cap E_{t}\cap E_{t,a^*}\cap G_{t,2}\cap G_{t,1}\right\}\\
        &\subseteq\left\{\frac{D^{(a^*)}_{\mathcal{G}}}{2} < r_t + 9\sigma_{a^*} \sqrt{\frac{2\log(t/\alpha)}{t}} + 7\sqrt{2\tilde{\sigma}_{a^*}} \left( \frac{\log(t/\alpha)}{t} \right)^{3/4} + \frac{8\log(t/\alpha)}{t}+\frac{K-1}{\sqrt{t}}+7\frac{\sqrt{(K-1)\log(t/\alpha)}}{t^{3/4}}\right\}\label{eq:fine_bound_of_n_1}.\\
    \end{align}
    Therefore, we can conclude that for all $t\geq \max\left\{t_4, t_1\right\}$, where 
    \begin{align}
        t_1:=\inf\left\{t\ge1:\frac{D^{(a^*)}_{\mathcal{G}}}{2} \ge r_t + 9\sigma_{a^*} \sqrt{\frac{2\log(t/\alpha)}{t}} + 7\sqrt{2\tilde{\sigma}_{a^*}} \left( \frac{\log(t/\alpha)}{t} \right)^{3/4} + \frac{8\log(t/\alpha)}{t}+\frac{K-1}{\sqrt{t}}+7\frac{\sqrt{(K-1)\log(t/\alpha)}}{t^{3/4}}\right\},
    \end{align}
    we have
    \begin{align}
        \mathbb{P}(D_{t,1})\leq\mathbb{P}(E_{t}^c)+\mathbb{P}(E_{t,a^*}^c)+\mathbb{P}(G^c_{t,2})+\mathbb{P}(G^c_{t,1}).
    \end{align}
    
    \begin{remark}
        The major difference is that we introduce the best data source in the definition of $t_1$. However, we need to pay an exploration cost, $\frac{K-1}{\sqrt{t}}$ and $7\frac{\sqrt{(K-1)\log(t/\alpha)}}{t^{3/4}}$, but they are second-order terms.
    \end{remark}
\end{proof}

\begin{lemma}[{\cite[Lemma 2]{shekhar2025nonparametrictwosampletestingbetting}}]\label{le:basic_ineq}
    The following relations are true:
    \begin{align}
        \tilde{\sigma}^2_{a^*}\leq\sigma^2_{a^*}+(D^{(a^*)}_\mathcal{G})^2,\quad\tilde{\sigma}_{a^*}\leq\sigma_{a^*}+D^{(a^*)}_\mathcal{G},\quad\text{and}\quad\sqrt{\tilde{\sigma}_{a^*}}\leq\sqrt{\sigma}_{a^*}+\sqrt{D^{(a^*)}_\mathcal{G}}.
    \end{align}
    where for each $k\in[K]$, $\tilde{\sigma}^2_k:=\sup_{g_k\in\mathcal{G}}\mathbb{E}[(g_k(X_{k,1})-g_k(X_{k,2}))^2]$ and $\gamma_k=\sup_{g_k\in\mathcal{G}}\text{Var}((g_k(X_{k,1})-g_k(X_{k,2}))^2)$.
\end{lemma}

\begin{lemma}\label{le:technical_lemma_for_stopping_time}
    Let $G_t:=G_{t,1}\cap G_{t,2}$ and 
    \begin{align}
        &G_{t,1}:=\left\{\frac{1}{t}\sum_{s=1}^{t}g^*_{1}(X_{a^*,1}(s))-g^*_{1}(X_{a^*,2}(s))\mathbbm{1}\left\{\delta_s=a^*\right\}\geq D^{(a^*)}_{\mathcal{G}}-\tilde{\sigma}_{a^*}\sqrt{\frac{4\log(t)}{t}}-\frac{2\log(t)}{3t}\right\},\\
        &G_{t,2}:=\left\{\frac{1}{t}\sum_{s=1}^t v^2_{a^*,s}\mathbbm{1}\left\{\delta_s=a^*\right\}\leq\tilde{\sigma}^2_{a^*}+\gamma_{a^*}\sqrt{\frac{4\log(t)}{t}}+\frac{2\log(t)}{3t}\right\}.
    \end{align}
    where $\tilde{\sigma}^2_{a^*}:=\sup_{g_{a^*}\in\mathcal{G}}\mathbb{E}[(g_{a^*}(X_{a^*,1})-g_{a^*}(X_{a^*,2}))^2]$ and $\gamma_{a^*}=\sup_{g_{a^*}\in\mathcal{G}}\text{Var}((g_{a^*}(X_{a^*,1})-g_{a^*}(X_{a^*,2})^2)$. $\mathbb{P}(G_t)\geq 1-\frac{2}{t^2}$.
\end{lemma}

\begin{proof}
    The proof is similar to \cite[Lemma 1]{shekhar2025nonparametrictwosampletestingbetting}. For notational simplicity, we define $v_{k,s}:=v^{(k)}_{s}(g_{k,s})$ for any $k\in[K]$, $s\ge1$, $g_{k,s}\in\mathcal{G}$ in this proof.

    Let us first show the upper bound of $\mathbb{P}(G_{t,1}^c)$. We define $a_s:=(g^*_{1}(X_{a^*,1}(s))-g^*_{1}(X_{a^*,2}(s))-D^{(a^*)}_{\mathcal{G}})\mathbbm{1}\left\{\delta_s=a^*\right\}$ and note that $\left\{a_s\right\}_{s\ge1}$ is a bounded martingale difference since $\mathbb{E}[a_s|F_{s-1}]=0$ and $a_s\leq 1$. Furthermore, we have $\sum_{s=1}^{t}\mathbb{E}[a_s^2|F_{s-1}]\leq t\tilde{\sigma}^2_{a^*}$. We can apply Lemma~\ref{le:bernstein} and get
    \begin{align}
        \mathbb{P}(G_{t,1}^c)=\mathbb{P}\left(\frac{1}{t}\sum_{s=1}^{t}a_s\leq-\delta_{1,t}\right)\leq\exp\left(-\frac{1}{2}\frac{t\delta_{1,t}}{\tilde{\sigma}^2_{a^*}+\frac{\delta_{1,t}}{3}}\right)\leq\frac{1}{t^2},
    \end{align}
    where the last inequality follows \cite[Lemma 1]{shekhar2025nonparametrictwosampletestingbetting} and $\delta_{1,t}:=\tilde{\sigma}_{a^*}\sqrt{\frac{4\log(t)}{t}}+\frac{2\log(t)}{3t}$.

    Then, for $\mathbb{P}(G_{t,2}^c)$, we can set $\beta_s:=\mathbb{E}[v_{a^*,s}^2|F_{s-1}]\mathbbm{1}\left\{\delta_s=a^*\right\}$. Then, following the same argument in \cite[Lemma 1]{shekhar2025nonparametrictwosampletestingbetting}, we have 
    \begin{align}
        &\beta_s\leq\tilde{\sigma}^2_{a^*}\quad\text{and}\\
        &\mathbb{E}[(v^2_{a^*,s}\mathbbm{1}\left\{\delta_s=a^*\right\}-\beta_s)^2|F_{s-1}]\leq\gamma^2_{a^*}.
    \end{align}
    Then, by Lemma~\ref{le:bernstein}, we get
    \begin{align}
        \mathbb{P}(G_{t,2}^c)&=\mathbb{P}\left(\frac{1}{t}\sum_{s=1}^t v^2_{a^*,s}\mathbbm{1}\left\{\delta_{s}=a^*\right\}-\tilde{\sigma}^2_{a^*}>\delta_{2,t}\right)\leq\mathbb{P}\left(\frac{1}{t}\sum_{s=1}^t \zeta_s>\delta_{2,t}\right)\leq\frac{1}{t^2}
    \end{align}
    where $\zeta_s:=v^2_{a^*,s}\mathbbm{1}\left\{\delta_{s}=a^*\right\}-\beta_s$ is a bounded martingale difference and $\delta_{2,t}:=\gamma_{a^*}\sqrt{\frac{4\log(t)}{t}}+\frac{2\log(t)}{3t}$.
\end{proof}

\end{document}